\documentclass[11pt,a4paper]{amsart}
\usepackage{mathrsfs}
\usepackage{syntonly}
\usepackage{amsmath}
\usepackage{amsthm}
\usepackage{amsfonts}
\usepackage{amssymb}
\usepackage{latexsym}
\usepackage{amscd,amssymb,amsopn,amsmath,amsthm,graphics,amsfonts,mathrsfs,accents,enumerate,verbatim,calc}
\usepackage[dvips]{graphicx}
\usepackage[colorlinks=true,linkcolor=red,citecolor=blue]{hyperref}
\usepackage[all]{xy}
\usepackage{enumitem}

\date{}
\hypersetup
{colorlinks=true,linkcolor=black}
\allowdisplaybreaks[4] \footskip=15pt
\renewcommand{\uppercasenonmath}[1]{}

\numberwithin{equation}{section} \theoremstyle{plain}
\newtheorem{lem}{Lemma}[section]
\newtheorem{cor}[lem]{Corollary}
\newtheorem{prop}[lem]{Proposition}
\newtheorem{thm}[lem]{Theorem}

\newtheorem{definition}[lem]{Definition}
\newtheorem{Ex}[lem]{Example}
\newtheorem{Quest}[lem]{Question}
\newtheorem{Property}[lem]{Property}
\newtheorem{Properties}[lem]{Properties}
\newtheorem{Subprops}{}[lem]
\newtheorem{Para}[lem]{}

\newtheorem{remark}[lem]{Remark}
\newtheorem{rem}[lem]{Remark}

\newtheorem*{ack*}{ACKNOWLEDGEMENTS}

\newcommand{\pf}{\noindent\begin {proof}}
\newcommand{\epf}{\end{proof}}

\newcommand{\s}{\stackrel}

\newcommand{\Ext}{\mbox{\rm Ext}}
\newcommand{\Hom}{\mbox{\rm Hom}}

\newcommand{\add}{\mbox{\rm add}}
\newcommand{\A}{\mathcal{A}}
\newcommand{\B}{\mathcal{B}}
\newcommand{\modcat}{\ensuremath{\mbox{{\rm -mod}}}}

\newcommand{\opp}{^{op}}
\newcommand{\pd}{{\rm proj.dim}}
\newcommand{\HA}{\mathcal{H}(\mathcal{A})}
\newcommand{\HB}{\mathcal{H}(\mathcal{B})}

\newcommand{\End}{{\rm End}}

\begin{document}
\begin{center}
{\Large  \bf  Resolving dualities and applications to homological invariants}

\vspace{0.5cm} Hongxing Chen and Jiangsheng Hu  \\
\end{center}

\bigskip
\centerline { \bf  Abstract}
\leftskip10truemm \rightskip10truemm \noindent
Dualities of resolving subcategories of finitely generated modules over Artin algebras are characterized as dualities with respect to Wakamatsu tilting bimodules. By restriction of these dualities to resolving subcategories of finitely generated modules with finite projective or Gorenstein-projective dimensions, Miyashita's duality and Huisgen-Zimmermann's correspondence on tilting modules as well as their Gorenstein version are obtained. Applications include
constructing triangle equivalences of derived categories of finitely generated Gorenstein-projective modules and showing the invariance of higher algebraic $K$-groups and semi-derived Ringel-Hall algebras of finitely generated Gorenstein-projective modules under tilting.
\\[2mm]

\leftskip0truemm \rightskip0truemm

{\footnotesize\tableofcontents\label{contents}}

\section {Introduction}
In homological algebra and representation theory, equivalences of additive categories of different types, such as derived equivalences and stable equivalences of Morita type (see \cite{Angeleri,Broue,Happel1988,keller,Rickard}), have been applied successfully to compare homological invariants or dimensions of relevant algebras and modules. For example, higher algebraic $K$-groups and finiteness of global and finitistic dimensions of rings are preserved under derived equivalences (see \cite{Daniel,Pan-and-xi}). As important subjects in tilting theory, tilting modules (see \cite{Brenner,Happel-Ringel,Happel1988,Miyashita}) not only provide a class of derived equivalences of algebras but also have some special properties related to resolving subcategories of module categories. Typical results include Brenner-Butler tilting theorem (see \cite{Brenner}) and a one-to-one correspondence between basic tilting modules and contravariant finite resolving subcategories of modules with finite projective dimension (see \cite{AR1991}). However, in contrast to covariant equivalences, dualities of categories have received far less attention. A classical result, due to Morita \cite{Anderson}, characterizes dualities of full subcategories of modules as dualities with respect to some balanced bimodule provided that the categories contain the appropriate regular modules. Moreover, the bimodule defines a Morita duality if and only if it is an injective cogenerator as one-sided modules. In the literature, there have been two generalizations of Morita dualities over Artin algebras. For tilting modules of finite projective dimension, Miyashita \cite{Miyashita} established dualities of some resolving subcategories of finitely generated modules with finite projective dimension; for Wakamatsu tilting modules (a generalization of tilting modules), Green, Reiten and Solberg \cite{GRS} established dualities of some resolving subcategories of finitely generated modules with relative coresolutions of arbitrary lengths. In these dualities, (Wakamatsu) tilting modules play the role of relative Ext-injective cogenerators, similar to injective cogenerators in Morita dualities. Recently, Huisgen-Zimmermann provided a converse of Miyashita's result: any dualities of resolving subcategories of finitely generated modules with finite projective dimension over Artin algebras are always afforded by tilting bimodules provided that the dualities are induced by strictly exact functors (see \cite{Huisgen,HS}).

In the paper, we characterize general dualities of resolving subcategories of finitely generated modules over Artin algebras as dualities with respect to a Wakamatsu tilting bimodule which is determined by the dualities. When the dualities can be restricted to resolving subcategories of modules with finite projective, Gorenstein-projective or semi-Gorenstein-projective dimension, we show that the Wakamatsu tilting bimodule is a tilting module. This leads to a Gorenstein version of  Miyashita's duality and Huisgen-Zimmermann's correspondence. As applications of dualities of resolving subcategories of finitely generated Gorenstein-projective modules, we show that both higher algebraic $K$-groups and semi-derived Ringel-Hall algebras of finitely generated Gorenstein-projective modules are preserved under tilting.

To state our results precisely, we first introduce some notation and definitions.

Throughout the paper, all algebras considered are Artin algebras and all modules are finitely generated left modules. Let $A$ be an algebra. Denote by $A\modcat$ the category of left $A$-modules and by $A\opp$ the opposite algebra of $A$. For a natural number $n$, the full subcategories of $A\modcat$ consisting of modules with projective, Gorenstein-projective and semi-Gorenstein-projective dimension at most $n$ are denoted by ${\mathcal{P}^{\leqslant n}(A)}$, ${\mathcal{GP}^{\leqslant n}(A)}$ and ${\mathcal{SGP}^{\leqslant n}(A)}$, respectively. For simplicity, we write ${\mathcal{P}(A)}$, ${\mathcal{GP}(A)}$ and ${\mathcal{SGP}(A)}$ for ${\mathcal{P}^{\leqslant 0}(A)}$, ${\mathcal{GP}^{\leqslant 0}(A)}$ and ${\mathcal{SGP}^{\leqslant 0}(A)}$, respectively.
As usual, the category of $A$-modules with finite projective dimension is denoted by $\mathcal{P}^{<\infty}(A)$.

Let $T$ be an $A$-module. We denote by $^{\perp}({_A}T)$ the full subcategory of $A\modcat$ consisting of modules $X$ with $\Ext_A^n(X,T)=0$ for all $n\geqslant 1$, and by $\mathcal{W}{(_{A}T)}$ the full subcategory of $^{\perp}({_A}T)$ consisting of modules $X$ which has an $\add(T)$-coresolution such that it stays exact after applying the functor $\Hom_{A}(-,T)$. Clearly, $\mathcal{W}{(_{A}A)}$ equals the category of Gorenstein-projective $A$-modules. Following \cite{Wakamatsu,Mantese}, $T$ is called a \emph{Wakamatsu tilting} $A$-module if both $A$ and $T$ belong to $\mathcal{W}{(_{A}T)}$. Further, if ${_A}T$ has finite projective dimension and ${_A}A$ has a finite $\add(T)$-coresolution, then $T$ is called a \emph{tilting} $A$-module. A close relationship between Wakamatsu tilting modules and tilting modules is predicted by the \emph{Wakamatsu tilting conjecture} which says that if a Wakamatsu tilting $A$-module has finite projective dimension, then it must be a tilting module (see \cite[Chapter IV]{Beli-Reiten}). The conjecture is of particular interest due to its strong connection with several long-standing homological conjectures such as the finitistic dimension conjecture and the generalized Nakayama conjecture (see \cite[Section 4]{Mantese}).

Let $B$ be another algebra. An $A$-$B$-bimodule ${_A}M{_B}$ is said to be \emph{faithfully balanced} if the canonical algebra homomorphisms $B\to{\End_A(M)\opp}$ and $A\to \End_{B\opp}(M)$ are isomorphisms; \emph{Wakamatsu tilting} (respectively, \emph{tilting}) if it is faithfully balanced and ${_A}M$ is Wakamatsu tilting (respectively, tilting). Observe that if $T$ is a Wakamatsu tilting $A$-module, then it is a Wakamatsu tilting $A$-$B$-bimodule
with $B=\End_A(T)\opp$, the endomorphism algebra of $T$ over $A$.

Motivated by dualities on tilting modules in \cite{Miyashita,Huisgen}, we introduce the definition of resolving dualities in abelian categories.

\begin{definition}\label{definition:2.1} Let $\mathcal{A}$ and $\mathcal{B}$ be abelian categories with enough projective objects, and let $\mathcal{C}$ and $\mathcal{D}$ be full subcategories of $\mathcal{A}$ and  $\mathcal{B}$, respectively.
Contravariant additive functors
$$F:\mathcal{C}\to\mathcal{D}\quad\mbox{and}\quad G:\mathcal{D}\to\mathcal{C}$$ are called \emph{inverse resolving dualities} between $\mathcal{C}$ and $\mathcal{D}$ if the following conditions hold:

\begin{enumerate}
\item $\mathcal{C}\subseteq \mathcal{A}$ and $\mathcal{D}\subseteq \mathcal{B}$ are resolving subcategories, that is, they contains all projective objects
and are closed under isomorphisms, extensions and kernels of epimorphisms.

\item  The compositions $G\circ F$ and $F\circ G$ are naturally isomorphic to the identity functors.

\item $F$ and $G$ are exact functors between $\mathcal{C}$ and $\mathcal{D}$ which are regarded as fully exact subcategories of $\mathcal{A}$ and $\mathcal{B}$, respectively.

\end{enumerate}
\end{definition}

Note that $F$ and $G$ between full subcategories of module categories satisfying Definition \ref{definition:2.1}(3) were called \emph{strictly exact} functors by Huisgen-Zimmermann in \cite{Huisgen}.  In our discussions, given that exact structures of resolving subcategories (as fully exact subcategories, see Section \ref{DCEC}) of abelian categories are highlighted, we prefer to use the terminology of exact functors between exact categories in Definition \ref{definition:2.1}(3).

Clearly, Morita duality and Miyachita's duality (see Theorem \ref{thm:2.3}) are resolving dualities. Our main result conveys that any resolving dualities between full subcategories of module categories of algebras are always afforded by Wakamatsu tilting bimodules, which are tilting modules when the dualities can be restricted to smaller resolving subcategories.

\begin{thm}\label{THM}
Let $A$ and $B$ be Artin algebras and let $\mathcal{C}\subseteq A\modcat$ and $\mathcal{D}\subseteq B^{op}\modcat$ be full subcategories. Suppose that $F:\mathcal{C}\to\mathcal{D}$ and $G:\mathcal{D}\to\mathcal{C}$ are inverse resolving dualities.  Then:

$(1)$ There exists a Wakamatsu tilting bimodule $_{A}T_{B}$ such that
\begin{enumerate}
\item[(a)] $T_{B}\cong F(_{A}A)$ and $_{A}T\cong G(B_{B})$;
\item[(b)] $F\cong\Hom_{A}(-,T)|_{\mathcal{C}}$ and  $G\cong\Hom_{B^{op}}(-,T)|_{\mathcal{D}}$;
\item[(c)] $\mathcal{C}\subseteq{\mathcal{W}{(_{A}T)}}$ and  $\mathcal{D}\subseteq{\mathcal{W}{(T_{B})}}$.
\end{enumerate}

$(2)$ The bimodule ${_A}T{_B}$ is a tilting bimodule if and only if $F$ and $G$ can be restricted to any one of inverse dualities of the following types:
$$\mathcal{C}\cap{\mathcal{P}^{< \infty }(A)}\simeq \mathcal{D}\cap{\mathcal{P}^{<\infty}(B^{op})},\quad\quad\;\; \mathcal{C}\cap{\mathcal{P}^{\leqslant n}(A)}\simeq \mathcal{D}\cap{\mathcal{P}^{\leqslant m}(B^{op})},$$
$$\mathcal{C}\cap{\mathcal{GP}^{\leqslant n}(A)}\simeq \mathcal{D}\cap{\mathcal{GP}^{\leqslant m}(B^{op})},\quad\quad  \;\;\mathcal{C}\cap{\mathcal{SGP}^{\leqslant n}(A)}\simeq \mathcal{D}\cap{\mathcal{SGP}^{\leqslant m}(B^{op})},$$
where $n$ and $m$ are some natural numbers.
\end{thm}

Theorem \ref{THM} generalizes Morita duality, Miyachita's duality and Huisgen-Zimmermann's correspondence;
see Corollary \ref{cor:3.8} and Remark \ref{generalization}. According to \cite[Theorem 1.5(d)]{GRS}, associated with a Wakamatsu tilting bimodule $_{A}T_{B}$, the functors $\Hom_{A}(-,T):A\modcat\to B^{op}\modcat$ and $\Hom_{B^{op}}(-,T):B^{op}\modcat \to A\modcat$ can be restricted to inverse resolving dualities $$(\diamondsuit)\quad\quad\mathcal{W}{(_{A}T)}\simeq\mathcal{W}{(T_{B})}.$$
It follows from Theorem \ref{THM}(1) that $(\diamondsuit)$ are ``\emph{maximal}" resolving dualities between subcategories of $A\modcat$ and $B^{op}\modcat$ which can be restricted to the resolving dualities between $\mathcal{C}$ and $\mathcal{D}$.

Combining Theorem \ref{THM} with $(\diamondsuit)$, we obtain a Gorenstein version of Miyachita's duality  and Huisgen-Zimmermann's correspondence (see Theorems \ref{thm:2.3} and \ref{thm:2.4}).

\begin{cor}\label{GV} Let $A$ and $B$ be Artin algebras.

$(1)$ If $_{A}T_{B}$ is a tilting bimodule, then the functors $\Hom_{A}(-,T)$ and $\Hom_{B^{op}}(-,T)$
can be restricted to inverse resolving dualities
${^{\perp}(_{A}T)}\cap{\mathcal{GP}^{\leqslant \ell}(A)}\simeq {^{\perp}(T_{B})}\cap{\mathcal{GP}^{\leqslant \ell}(B^{op})}$,
where $\ell$ denotes the projective dimension of ${_A}T$.

$(2)$ Let $\mathcal{C}\subseteq A\modcat$ and $\mathcal{D}\subseteq B^{op}\modcat$ be full subcategories.
Suppose that  there are  inverse resolving dualities between $\mathcal{C}$ and $ \mathcal{D}$.
If {$\mathcal{GP}(A)\subseteq\mathcal{C}\subseteq {\mathcal{GP}^{\leqslant n }(A)}$ and $\mathcal{GP}(B^{op})\subseteq\mathcal{D}\subseteq {\mathcal{GP}^{\leqslant m}(B^{op})}$} for some natural numbers $n$ and $m$, then there is a tilting bimodule $_{A}T_{B}$ such that
$$\mathcal{C}={^{\perp}(_{A}T)}\cap{\mathcal{GP}^{\leqslant \ell}(A)}\quad\mbox{and}\quad  \mathcal{D}={^{\perp}(T_{B})}\cap{\mathcal{GP}^{\leqslant \ell}(B^{op})},$$
where $\ell$ denotes the projective dimension of ${_A}T$.
\end{cor}

Remark that if both $A$ and $B^{op}$ have finite finitistic dimension, then ${\mathcal{GP}^{\leqslant n }(A)}$ and ${\mathcal{GP}^{\leqslant m}(B^{op})}$ in Corollary \ref{GV}(2) can be replaced with ${\mathcal{GP}^{<\infty}(A)}$ and ${\mathcal{GP}^{<\infty}(B^{op})}$, respectively. This follows from the fact that the finitistic dimension of an algebra $A$ equals the supremum of Gorenstein-projective dimensions of all those $A$-modules which have finite
Gorenstein-projective dimensions (see \cite[Lemma 4.4]{Xi}).

Finally, we apply the resolving dualities in Corollary \ref{GV} to compare Gorenstein derived categories and homological invariants of algebras linked by tilting bimodules. Similar applications of Miyachita's duality to homotopy categories and algebraic $K$-groups associated with projective modules are given in Corollary \ref{Projective}.

Let $\mathcal{E}$ be an exact category in the sense of Quillen. For $*\in\{\emptyset, +,-,b\}$, the \emph{$*$-derived category} of
$\mathcal{E}$, denoted by $\mathscr{D}^*(\mathcal{E})$, is defined to be the Verdier quotient of $*$-homotopy category $\mathscr{K}^*(\mathcal{E})$ of $\mathcal{E}$ by the full triangulated subcategory of strictly exact complexes (see Section \ref{DCEC}). When $\mathcal{E}$ is a small category (that is, the isomorphism classes of objects of $\mathcal{E}$ is a set), we denote by $K_n(\mathcal{E})$ the \emph{$n$-th algebraic $K$-group} of $\mathcal{E}$ for each $n\in\mathbb{N}$ (see \cite{Quillen}). In particular,  $K_n(\mathcal{GP}(A))$ is called the \emph{$n$-th Gorenstein algebraic $K$-group} of $A$. When $\mathcal{E}$ is a weakly $1$-Gorenstein exact category with finite morphism spaces and finite extension spaces, the \emph{semi-derived Ringel-Hall algebra} of $\mathcal{E}$, denoted by $\mathcal{SDH}(\mathcal{E})$, was defined in \cite{Lu2022} (see also \cite{Bridgeland,Gorsky,LP21} for some cases) and applied to stimulate further interactions between communities on Hall algebras and on quantum symmetric pairs. Given a finite-dimensional algebra $A$ over a finite field, $\mathcal{GP}(A)$ is a weakly $1$-Gorenstein exact category. If $A$ is $1$-Gorenstein, then $A\modcat$ is also weakly $1$-Gorenstein. In this case, $\mathcal{SDH}(A\modcat)$ is called the \emph{semi-derived Ringel-Hall algebra} of $A$ and simply denoted by $\mathcal{SDH}(A)$.

\begin{cor}\label{Cor-1} Let $A$ and $B$ be Artin algebras and $_{A}T_{B}$ a tilting bimodule. Then:
\begin{enumerate}
\item  There is a triangle equivalence $\mathscr{D}(\mathcal{GP}(A))\simeq{\mathscr{D}(\mathcal{GP}(B))}$ which
    can be restricted to an equivalence $\mathscr{D}^{*}(\mathcal{GP}(A))\simeq \mathscr{D}^{*}(\mathcal{GP}(B))$ for any $*\in\{+,-,b\}$.

\item $K_{n}(\mathcal{GP}(A))\simeq{K_{n}(\mathcal{GP}(B))}$ for any $n\in{\mathbb{N}}.$

\item Suppose that $A$ is a finite-dimensional algebra over a finite field and ${_A}T$ has projective dimension at most $1$. Then
$\mathcal{SDH}(\mathcal{GP}(A))\cong \mathcal{SDH}(\mathcal{GP}(B))$ as algebras.
\end{enumerate}
\end{cor}

Corollary \ref{Cor-1}(3) refines \cite[Corollary A23]{Lu2022} in which both $A$ and $B$ are required to be $1$-Gorenstein algebras. Moreover, our proof is different from the one given in \cite{Lu2022}.  Combining Corollary \ref{Cor-1}(3) with Proposition \ref{prop:4.7}, we can also provide a new proof of \cite[Theorem A15]{Lu2022} which says that $\mathcal{SDH}(A)\cong\mathcal{SDH}(B)$ as algebras for a tilting bimodule ${_A}T{_B}$ over $1$-Gorenstein algebras $A$ and $B$; see Corollary \ref{corollary:4.9'} and Remark \ref{Different} for more explanation.

\smallskip
The article is organized as follows. In Section 2, we introduce notation, definitions and basic facts needed for proofs.
In particular, we recall derived categories of exact categories, (co)resolutions dimensions, Miyashita's duality and Huisgen-Zimmermann's correspondence. In Section 3, we study homological properties of Wakamatsu tilting modules (see Corollary \ref{cor:3.5}), and show Theorem \ref{THM} and Corollary \ref{GV}. In Section 4, we apply resolving dualities to establish triangle equivalences of derived categories and show the invariance of higher algebraic $K$-groups and semi-derived Ringel-Hall algebras under tilting. This implies Corollary \ref{Cor-1}.

\section{Preliminaries}\label{preliminaries}

In this section we give some definitions in our terminology and collect facts which are used in the paper.

\subsection{Notation and definitions}

Let $A$ be an Artin $R$-algebra, that is, $R$ is a commutative Artin ring and $A$ is an $R$-algebra which is finitely generated as an $R$-module. As usual, $A^{op}$ stands for the opposite algebra of $A$. Denote by $A$-mod the category of finitely generated left $A$-modules.  The kernel, image and cokernel of a homomorphism $f$ in $A$-mod are denoted by ${\rm Ker}(f)$, ${\rm Im}(f)$ and ${\rm Coker}(f)$, respectively.

Let $T$ be an $A$-module. We define
$$^{\perp}({_A}T):=\{M\in A{\text-}{\rm mod}\ |\ \Ext_{A}^i(M,T)=0\ for \ all\ i\geqslant1 \}.$$
Similarly, $({_A}T)^{\perp}$ can be defined.  Denote by add$(_{A}T)$ the full subcategory of $A$-mod consisting of direct summands of finite direct sums of copies of $T$. The $n$-th syzygy of ${_A}T$ is denoted by $\Omega^{n}_{A}(T)$ for each $n\in\mathbb{N}$.

\subsection{Derived categories of exact categories}\label{DCEC}
An exact category $\mathcal{E}$  (in the sense of Quillen) is by definition an additive category endowed with a class of conflations closed under isomorphism and satisfying certain axioms (see \cite{Quillen,keller} for details). When the additive category is abelian, the class of conflations coincides with the class of short exact sequences. An additive functor $F:\mathcal{E}\to \mathcal{E'}$ between exact categories $\mathcal{E}$ and $\mathcal{E'}$ is said to be \emph{exact} if it
sends the conflations in $\mathcal{E}$ to the ones in $\mathcal{E'}$.

Let $\mathcal{E}$ be an exact category and $\mathcal{F}$ a full subcategory of $\mathcal{E}$. If $\mathcal{F}$ is closed under extensions in $\mathcal{E}$, then $\mathcal{F}$, endowed with the conflations in $\mathcal{E}$ having their terms in $\mathcal{F}$ , is an exact category, and the inclusion $\mathcal{F}\subseteq\mathcal{E}$ is
a fully faithful exact functor. In this case, $\mathcal{F}$ is called a \emph{fully exact subcategory} of $\mathcal{E}$ (see \cite[Section 4]{keller}). In an abelian category $\mathcal{A}$ with enough projective objects, resolving subcategories are fully exact subcategories.
Throughout the paper, we always regard resolving subcategories of $\mathcal{A}$ as exact categories.

Let $\A$ be an abelian category and $\mathcal{E}$ a fully exact subcategory of $\A$.  Denote by $\mathscr{C}(\mathcal{E})$ the category of complexes over $\mathcal{E}$.
A complex $X\in{\mathscr{C}(\mathcal{E})}$
is said to be \emph{strictly exact} if it is an exact complex over $\A$ and all of its boundaries belong to $\mathcal{E}$. Let $\mathscr{K}_{ac}(\mathcal{E})$ be the full
subcategory of $\mathscr{K}(\mathcal{E})$ consisting of those complexes which are isomorphic to strictly exact complexes. Then
$\mathscr{K}_{ac}(\mathcal{E})$ is a full triangulated subcategory of $\mathscr{K}_{ac}(\mathcal{E})$ closed under direct summands. The \emph{unbounded derived
category} of $\mathcal{E}$, denoted by $\mathscr{D}(\mathcal{E})$, is defined to be the Verdier quotient of $\mathscr{K}(\mathcal{E})$ by $\mathscr{K}_{ac}(\mathcal{E})$. Similarly, the
bounded-below, bounded-above and bounded derived categories $\mathscr{D}^{+}(\mathcal{E})$, $\mathscr{D}^{-}(\mathcal{E})$ and $\mathscr{D}^{b}(\mathcal{E})$ can be defined.

In Section \ref{K-theory}, we need the following result (see \cite[Section 11]{keller} and \cite[Proposition A.5.6]{Positselski}).

\begin{lem}\label{lemma:2.3'} Let $\mathcal{F}$ and $\mathcal{E}$ be fully exact subcategories of $\A$ with $\mathcal{F} \subseteq \mathcal{E}$. Assume that $\mathcal{E}$ is closed under direct summands in $\A$ and the following two conditions hold:
\begin{enumerate}
\item[(a)] For an exact sequence $0 \to X \to Y \to Z \to 0$ in $\A$ with $X\in{\mathcal{E}}$, if $Y,Z\in{\mathcal{F}}$, then $X\in{\mathcal{F}}$.

\item[(b)] There is a natural number $n$ such that, for each object $E\in{\mathcal{E}}$, there is a long exact sequence in $\A$
$$0\to F_{n}\s{f_{n}}\to\cdots\to F_{1}\s{f_1}\to F_{0}\s{f_0}\to E\to 0$$
\noindent with $F_{i}\in{\mathcal{F}}$ and ${\rm Im}(f_i)\in{\mathcal{E}}$ for all $0 \leqslant i \leqslant n$.
\end{enumerate}

\noindent Then the inclusion $\mathcal{F} \subseteq \mathcal{E}$
induces a triangle equivalence $\mathscr{D}(\mathcal{F})\to \mathscr{D}(\mathcal{E})$ which can be restricted to an equivalence $\mathscr{D}^{*}(\mathcal{F})\to \mathscr{D}^{*}(\mathcal{E})$ for any $*\in\{+,-,b\}$.
\end{lem}

\subsection{(semi-)Gorenstein-projective modules and $\mathcal{X}$-(co)resolutions}
A complete projective resolution over $A$ is by definition a (double infinite) exact complex
$P^{\bullet}: \cdots \to P^{-1}\s{d_P^{-1}}\to P^{0}\s{d_P^{0}}\to P^{1}\to\cdots$
of projective $A$-modules such that the complex $\Hom_{A}(P^{\bullet},A)$, obtained by applying $\Hom_A(-,A)$ to $P^{\bullet}$, is again exact.

\begin{definition} \emph{\cite{EJ1995,EJ}} An $A$-module $M$ is said to be \emph{Gorenstein-projective} if there is a complete projective resolution $P^{\bullet}$ over $A$ such that $M$ is isomorphic to the image of $d_P^{-1}: P^{-1}\to P^0$.
\end{definition}

Gorenstein-projective modules were called modules of G-dimension zero in \cite{AB1969} or totally reflexive modules in \cite[Section 2]{AM2002}.
A generalization of Gorenstein-projective modules is the following.

\begin{definition}\cite{RZ1} An $A$-module $M$ is said to be \emph{semi-Gorenstein-projective} provided that $\Ext_{A}^{i}(M,A)=0$ for all $i\geqslant1$.
\end{definition}

Denote by $\mathcal{GP}(A)$ and $\mathcal{SGP}(A)$ the categories of Gorenstein-projective and semi-Gorenstein-projective $A$-modules, respectively.
Then $\mathcal{GP}(A)$ and $\mathcal{SGP}(A)$ are resolving subcategories of $A\modcat$. Moreover,
$\mathcal{GP}(A)\subseteq \mathcal{SGP}(A)$, but the converse of the inclusion is not true in general (see \cite{RZ1} for examples).

Let $\mathcal{X}$ be a full subcategory of $A$-mod. An \emph{$\mathcal{X}$-resolution} of an $A$-module $M$ is by definition
a long exact sequence of $A$-modules  $\cdots\to X_2\to X_{1}\to X_{0}\to M\to 0$
with $X_{i}\in{\mathcal{X}}$ for all $i$. The length of the resolution is defined to be the supremum of $i$ such that  $X_i\neq0$.
Now, the \emph{$\mathcal{X}$-resolution dimension} of $M$, denote by res.dim$_{\mathcal{X}}(M)$,
is defined to be the minimal natural number $n$ such that $M$ has an $\mathcal{X}$-resolution
of length $n$, or $\infty$ if no such $n$ exists.
The concept of $\mathcal{X}$-coresolutions and  $\mathcal{X}$-coresolution dimension of $M$ can be defined dually.

The following result is known in the literature.

\begin{lem}\label{lemma:2.1} Let $\mathcal{X}$ be a resolving subcategory of $A$-{\rm mod} and $M$ an $A$-module. The following are equivalent for a natural number $n$:
\begin{enumerate}
\item {\rm res.dim}$_{\mathcal{X}}(_{A}M)\leqslant n$.
\item $\Omega^n_{A}(M)\in{\mathcal{X}}$.
\item For any $\mathcal{X}$-resolution $\cdots\to X_2\to X_{1}\to X_{0}\to M\to 0$ of $M$, ${\rm Ker}(X_{n-1}\to X_{n-2})\in \mathcal{X}$.
\end{enumerate}
\end{lem}
Associated with $\mathcal{X}$, there are
$$\mathcal{X}^{\leqslant n}(A):=\{M\in{A\textrm{-mod}} \ | \ \textrm{res.dim}_{\mathcal{X}}(_{A}M)\leqslant n\},$$
$$\mathcal{X}^{<\infty}(A):=\{M\in{A\textrm{-mod}} \ | \ \textrm{res.dim}_{\mathcal{X}}(_{A}M)<\infty\}.$$
Clearly, if $\mathcal{X}$ is a resolving subcategory of $A$-{\rm mod}, then so are $\mathcal{X}^{\leqslant n}(A)$ and $\mathcal{X}^{<\infty}(A)$.
Denote by $\mathcal{P}(A)$  the category of projective $A$-modules.  For simplicity, we set
$$\mathcal{P}^{\leqslant n}(A):=\mathcal{P}(A)^{\leqslant n}(A),\quad \mathcal{P}^{<\infty}(A):=\mathcal{P}(A)^{<\infty}(A),$$
$$\mathcal{GP}^{\leqslant n}(A):=\mathcal{GP}(A)^{\leqslant n}(A),\quad \mathcal{GP}^{<\infty}(A):=\mathcal{GP}(A)^{<\infty}(A),$$
$$\mathcal{SGP}^{\leqslant n}(A):=\mathcal{SGP}(A)^{\leqslant n}(A),\quad \mathcal{SGP}^{<\infty}(A):=\mathcal{SGP}(A)^{<\infty}(A).$$

%and For $n\in\mathbb{N}$,
%we write $\mathcal{P}^{\leqslant n}(A)$ for $\mathcal{P}(A)^{\leqslant n}(A)$, and
%$\mathcal{P}^{<\infty}(A)$ for $\mathcal{P}(A)^{<\infty}(A)$.
%
%\subsection{Gorenstein-projective and semi-Gorenstein projective modules}

\subsection{Resolving dualities induced by tilting modules}
An $A$-module $T$ is called an \emph{$n$-tilting} module (see \cite{Brenner,Happel-Ringel,Happel1988,Miyashita}) if the following conditions are satisfied:
\begin{enumerate}
\item[(T1)]
 $T\in{\mathcal{P}^{\leqslant n}(A)}$, that is,
the projective dimension of $_{A}T$, denoted by proj.dim$(_{A}T)$, is at most $n$.

\item[(T2)] $\Ext_{A}^{j}(T,T)=0$ for all $j\geqslant1$.

\item[(T3)] There exists an exact sequence of $A$-modules
$0\to {_{A}A}\to T_{0}\to\cdots\to T_{n}\to 0$
with $T_{i}\in{\textrm{add}(_{A}T)}$ for all $0\leqslant i\leqslant n$.
\end{enumerate}

A module $_{A}T$ is said to be \emph{tilting} if it is $n$-tilting for some natural number $n$. If, in addition,
$\mathcal{P}^{<\infty}(A)\subseteq{^{\perp}(_{A}T)}$, then $_{A}T$ is said to be \emph{strong tilting} (see \cite{AR1991}).

Given a tilting module $_{A}T$ with $B:=\textrm{End}_{A}(T)^{op}$, we see that $T_{B}$ is
also a tilting module and $A$ is isomorphic to $\textrm{End}_{B^{op}}(T)$ as algebras.

The following theorem is readily deduced from \cite[Theorem 3.5]{Miyashita}.

\begin{thm}{\rm (Miyashita's duality)}\label{thm:2.3} For a tilting bimodule $_{A}T_{B}$, let
\vspace{2mm}
\begin{center}{$\mathcal{C}:={^{\bot}(_{A}T)}\cap{\mathcal{P}^{<\infty}(A)}$ and $\mathcal{D}:={^{\bot}(T_{B})}\cap{\mathcal{P}^{<\infty}(B^{op})}$.}
\end{center}
\vspace{2mm}
Then the restricted Hom-functors $\Hom_{A}(-,T)|_{\mathcal{C}}:\mathcal{C}\to\mathcal{D}$ and $\Hom_{B^{op}}(-,T)|_{\mathcal{D}}:\mathcal{D}\to\mathcal{C}$
are inverse resolving dualities. Further, if $_{A}T_{B}$ is {strong tilting}, then there are inverse resolving dualities
$\mathcal{P}^{<\infty}(A)\simeq{\mathcal{P}^{<\infty}(B^{op})}$.
\end{thm}

Recently, Huisgen-Zimmermann has supplemented Miyashita's duality by showing that resolving dualities between subcategories of $\mathcal{P}^{<\infty}(A)$ and $\mathcal{P}^{<\infty}(B^{op})$ are always afforded by tilting bimodules (see \cite[Theorem 1]{Huisgen}).

\begin{thm}{\rm (Huisgen-Zimmermann's correspondence)} \label{thm:2.4}Let $\mathcal{C} \subseteq \mathcal{P}^{<\infty}(A)$ and $\mathcal{D} \subseteq \mathcal{P}^{<\infty}(B^{op})$ be resolving
subcategories of $A\textrm{-}{\rm mod}$ and $B^{op}\modcat$, respectively.
If $F:\mathcal{C}\to\mathcal{D}$ and $G:\mathcal{D}\to\mathcal{C}$ are inverse resolving dualities, then there exists a tilting bimodule $_{A}T_{B}$ such that
\begin{enumerate}
\item[(a)] $F\cong\Hom_{A}(-,T)|_{\mathcal{C}}$ and  $G\cong\Hom_{B^{op}}(-,T)|_{\mathcal{D}}$,

\item[(b)] $\mathcal{C}={^{\bot}(_{A}T)\cap{\mathcal{P}^{<\infty}(A)}}$ and  $\mathcal{D}={^{\bot}(T_{B})}\cap{\mathcal{P}^{<\infty}(B^{op})}$, and

\item[(c)] $\mathcal{C}$ consists of those modules $M$ which have finite ${\rm add}(_{A}T)$-coresolutions and $\mathcal{D}$ consists of those modules $N$ which have finite  ${\rm add}(T_{B})$-coresolutions.
\end{enumerate}
\end{thm}

By Theorem \ref{thm:2.4}(c), we have ${^{\bot}(_{A}T)\cap{\mathcal{P}^{<\infty}(A)}}={^{\bot}(_{A}T)}\cap\mathcal{P}^{\leqslant \ell}(A)$ and ${^{\perp}(T_{B})}\cap{\mathcal{P}^{<\infty}(B^{op})}={^{\perp}(T_{B})}\cap{\mathcal{P}^{\leqslant \ell}(B^{op})}$, where
$\ell$ denotes the projective dimension of ${_A}T$.

\section{Correspondence between Wakamatsu tilting modules and resolving dualities}\label{proof}
In the section, we discuss resolving subcategories of module categories related to Wakamatsu tilting modules
and establish relationships between resolving dualities and Wakamatsu tilting modules.
In case that these dualities can be restricted to resolving subcategories of modules with finite projective, Gorenstein-projective or semi-Gorenstein-projective dimension, the associated Wakamatsu tilting modules are shown to be tilting. In particular, we show Theorem \ref{THM} and Corollary \ref{GV}.

\subsection{Basic facts on Wakamatsu tilting modules}
Let $T$ be an $A$-module. We denote by cogen$^{*}(_{A}T)$  the full subcategory of $A$-mod consisting of modules $M$
which admits an exact sequence of $A$-modules
$0\to M\to T_{0}\to T_{1}\to T_2\to \cdots$ with $T_{i}\in{\textrm{add}(_{A}T)}$ for all $i\geqslant 0$ such that applying the functor $\Hom_{A}(-,T)$ to the sequence still yields an exact sequence. Equivalent characterizations of cogen$^{*}(_{A}T)$ are given in the following result (for example, see
\cite[Lemmas 2.2 and 2.4]{Mabiao}).

\begin{lem}\label{lemma:3.1}  Let $B={\rm End}_{A}(T)^{op}$. For an $A$-module $M$, the following statements are equivalent.
\begin{enumerate}
\item $M\in{{\rm cogen}^{*}(_{A}T)}$.

\item The map $\sigma_{M}:M\to \Hom_{B^{op}}(\Hom_{A}(M,T),T)$, $m\mapsto[f\mapsto f(m)]$ for $m\in M$ and $f\in\Hom_{A}(M,T)$ is an isomorphism and $\Hom_{A}(M,T)\in{^{\perp}(T_{B})}$.

\item  The canonical maps $$\Ext^{i}_{A}(N,M)\to \Ext^{i}_{B\opp}(\Hom_{A}(M,T),\Hom_{A}(N,T))$$
induced by the functor $\Hom_A(-,T)$ are isomorphisms for all $N\in{^{\perp}(_{A}T)}$ and $i\geqslant 0$.
\end{enumerate}
\end{lem}

In the introduction, we have defined $${\mathcal{W}{(_{A}T)}}:={^{\perp}(_{A}T)}\cap {\rm cogen}^{*}(_{A}T).$$
Clearly, $T\in {\mathcal{W}{(_{A}T)}}$ if and only if $\Ext_A^i(T,T)=0$ for all $i\geqslant 1$. Thanks to Lemma \ref{lemma:3.1}(3),
the functor $\Hom_A(-,T):A\modcat\to B^{op}\modcat$ can be restricted to a fully faithful functor ${\mathcal{W}{(_{A}T)}}\to B^{op}\modcat$ which preserves extension groups of modules. However, in general, ${\mathcal{W}{(_{A}T)}}$ is not a resolving subcategory of $A$-mod since it
may not contain projective $A$-modules. Following  \cite{Wakamatsu}, an $A$-module $T$ is said to be \emph{Wakamatsu tilting} if
$T\in {\mathcal{W}{(_{A}T)}}$  and $_{A}A\in{\mathcal{W}{(_{A}T)}}$.
This is also equivalent to the following two conditions:

$(1)$ $\textrm{End}_{B\opp}(T)\cong A$, where $B=\textrm{End}_{A}(T)^{op}$;

$(2)$ $\Ext_{A}^{i}(T,T)=0=\Ext_{B\opp}^{i}(T,T)$ for all $i\geqslant1$.

\noindent By these conditions, if $_{A}T$ is Wakamatsu tilting, then $T_{B}$ is also Wakamatsu tilting. So, we can say $T$ is a Wakamatsu tilting $A$-$B$-bimodule.

Now, we collect some basic properties of Wakamatsu tilting modules.

\begin{lem}{\rm(}\cite[Proposition 2.11]{Mantese}{\rm)}\label{proposition:3.2}
Let $T$ be a Wakamatsu tilting $A$-module. Then:

(1) $\mathcal{W}{(_{A}T)}$ is a resolving subcategory of $A$-{\rm mod}.

(2) $\mathcal{W}{(_{A}T)}\cap{\mathcal{W}{(_{A}T)}}^{\perp}={\rm add}(_{A}T)$ and
 $^{\perp}({\mathcal{W}{(_{A}T)}}^{\perp})=\mathcal{W}{(_{A}T)}$.

(3) $T$ is an injective cogenerator for $\mathcal{W}{(_{A}T)}$, that is, for any $X\in \mathcal{W}{(_{A}T)}$,
there is an exact sequence $0\to X\to I_0\to X_1\to 0$ in $A\modcat$
such that $I_0\in\add(_AT)$ and  $X_1\in \mathcal{W}{(_{A}T)}$.
\end{lem}

For simplicity, sometimes we write $(X,Y)$ for $\Hom_{A}(X,Y)$ in our proofs for $A$-modules $X$ and $Y$.

\begin{lem}\label{prop:3.3} Let $_{A}T$ be a Wakamatsu tilting module. Then:

(1) $^{\perp}(_{A}T)\cap{\mathcal{P}^{\leqslant n}(A)}=
\mathcal{W}{(_{A}T)}\cap{\mathcal{P}^{\leqslant n}(A)}$ for any $n\geqslant0$;

(2) $_{A}T\in{\mathcal{P}^{<\infty}(A)}$ if and only if $_{A}T\in{\mathcal{GP}^{<\infty}(A)}$ and $\mathcal{GP}(A)\subseteq{\mathcal{W}{(_{A}T)}}$;

(3) If $_{A}T\in{\mathcal{P}^{<\infty}(A)}$, then
 $\mathcal{W}{(_{A}T)}\cap{\mathcal{GP}^{\leqslant n }(A)}={^{\perp}(_{A}T)}\cap{\mathcal{GP}^{\leqslant n }(A)}$ for any $n\geqslant0$.

(4) If $_{A}T$ is a tilting module of projective dimension $\ell$, then $\mathcal{W}{(_{A}T)}\subseteq{\mathcal{GP}^{\leqslant \ell}(A)}$.
\end{lem}
\begin{proof} (1) It suffices to prove $^{\perp}(_{A}T)\cap{\mathcal{P}^{\leqslant n}(A)}\subseteq
{{\rm cogen}^{*}(_{A}T)}$. Let $M\in{^{\perp}(_{A}T)}\cap{\mathcal{P}^{\leqslant n}(A)}$. Then there exists an exact sequence in $A$-mod $$0\to P_{n}\to\cdots\to P_1\to P_0\to M\to 0,$$
where $P_{i}$ is projective for all $0\leqslant i\leqslant n$. This yields the following exact sequence in $B^{op}\modcat$:
$$0\to \Hom_{A}(M,T)\to \Hom_{A}(P_{0},T)\to  \Hom_{A}(P_{1},T)\to \cdots \to \Hom_{A}(P_{n},T)\to 0.$$
Since $\Hom_{A}(P_{i},T)\in{\textrm{add}(T_{B})}\subseteq{^{\perp}(T_{B})}$ for all $0\leqslant i\leqslant n$ and since  ${^{\perp}(T_{B})}$ is a resolving subcategory of $B^{op}\modcat$, we have $\Hom_{A}(M,T)\in{^{\perp}(T_{B})}$. It follows that
there is a commutative diagram with exact rows in $A$-mod:
$$\xymatrix{
0\ar[r] &P_{n}\ar[d]^{\sigma_{P_{n}}}\ar[r]&\cdots \ar[r]&P_{0}\ar[r]\ar[d]^{\sigma_{P_{0}}}&M\ar[r]\ar[d]^{\sigma_{M}}&0\\
0\ar[r] &((P_{n},T),T)\ar[r]&\cdots \ar[r]&((P_{0},T),T)\ar[r]&((M,T),T)\ar[r]&0.}$$
 Since $\sigma_{P_{i}}$ is an isomorphism for all $0\leqslant i\leqslant n$, we see that $\sigma_{M}$ is an isomorphism.
 Thus $M\in{{\rm cogen}^{*}(_{A}T)}$ by Lemma \ref{lemma:3.1}.

(2) Assume that $_{A}T\in{\mathcal{GP}^{<\infty}(A)}$ and $\mathcal{GP}(A)\subseteq{\mathcal{W}{(_{A}T)}}$. Since $_{A}T\in{\mathcal{GP}^{<\infty}(A)}$, we see from \cite[Lemma 2.17]{cfh} that there is an exact sequence
$0\to T\to H\to G\to 0 $ in $A$-mod,  where $H\in{\mathcal{P}^{<\infty}(A)}$ and $G\in{\mathcal{GP}(A)}$. By $\mathcal{GP}(A)\subseteq{\mathcal{W}{(_{A}T)}}$,
the exact sequence $0\to T\to H\to G\to 0$ splits. Thus $_{A}T$ is isomorphic to a direct summand of $H$, which forces $T\in{\mathcal{P}^{<\infty}(A)}$.

Conversely, assume $_{A}T\in{\mathcal{P}^{<\infty}(A)}$. Let $M$ be an $A$-module in $\mathcal{GP}(A)$. There exists a complete projective resolution over $A$
$$P^{\bullet}:\cdots \to P^{-1}\s{d^{-1}}\to P^{0}\s{d^{0}}\to P^{1}\s{d^{1}}\to\cdots$$
such that $M\cong \textrm{Im}(d^{-1})$. Set $K^{i}:=\textrm{Im}(d^{i})$ for all $i\in{\mathbb{Z}}$. Then $K^{i}\in{\mathcal{GP}(A)}$. Since $\mathcal{GP}(A)\subseteq{^{\perp}(_{A}T)}$,  the complex $\Hom_{A}(P^{\bullet},T)$ is again exact. Hence we have the following commutative diagrams
$$\xymatrix{
0\ar[r] &M\ar[d]^{\sigma_{M}}\ar[r]&P^{0} \ar[d]^{\cong}\ar[r]&K^{1}\ar[r]\ar[d]^{\sigma_{K^{1}}}&0\\
0\ar[r]&((M,T),T)\ar[r]&((P^{0},T),T)\ar[r]&(((K^{1},T),T)}$$
and
$$\xymatrix{
0\ar[r] &K^{1}\ar[d]^{\sigma_{K^{1}}}\ar[r]&P^{1} \ar[d]^{\cong}\ar[r]&K^{2}\ar[r]\ar[d]^{\sigma_{K^{2}}}&0\\
0\ar[r]&((K^{1},T),T)\ar[r]&((P^{1},T),T)\ar[r]&((K^{2},T),T).}$$
This implies that both $\sigma_{K^{1}}$ and $\sigma_{M}$ are injective, and therefore $\sigma_{M}$ is an isomorphism by the snake lemma. Similarly, we can show that $\sigma_{K^{i}}:K^{i}\to((K^{i},T),T)$ is an isomorphism for any $i\in{\mathbb{Z}}$. Thus
$\Ext^{i}_{B}(\Hom_{A}(M,T),T)=0$ for all $i\geqslant1$. By Lemma \ref{lemma:3.1}, $M\in{{\rm cogen}^{*}(_{A}T)}$.
Since $M\in{{^{\perp}(_{A}T)}}$, we have $M\in{\mathcal{W}{(_{A}T)}}$.

(3) It suffices to show ${^{\perp}(_{A}T)}\cap{\mathcal{GP}^{\leqslant n }(A)}\subseteq\mathcal{W}{(_{A}T)}$.
Let $M\in{^{\perp}(_{A}T)}\cap{\mathcal{GP}^{\leqslant n }(A)}$. By  \cite[Lemma 2.17]{cfh}, there exists an exact sequence $0\to M\to H\to G\to 0$ in $A$-mod with {\rm proj.dim$(_{A}H)\leqslant n$} and $G\in{\mathcal{GP}(A)}$. Since $G, M\in{^{\perp}(_{A}T)}$,  we have $H\in {^{\perp}(_{A}T)}$. This forces   $H\in{\mathcal{W}{(_{A}T)}}$ by $(1)$. Since $G\in{\mathcal{W}{(_{A}T)}}$ by (2), $M\in{\mathcal{W}{(_{A}T)}}$ by Lemma \ref{proposition:3.2}.

(4) Let $M\in\mathcal{W}{(_{A}T)}$. There exists a long exact sequence
$0\to M\to X_{0}\to X_{1}\to \cdots$
in $A$-mod with $X_{i}\in{\textrm{add}(_{A}T)}$ for all $i\geqslant0$ such that it stays exact after applying $\Hom_{A}(-,T)$. Since proj.dim$(_{A}T)=\ell$, we see from the horseshoe lemma that there exists an exact sequence in $A\modcat$
$$0\to \Omega^{\ell}_{A}(M)\to Q_{0}\to Q_{1}\to \cdots,$$
where $Q_{i}$ is projective for each $i\geqslant0$. Recall that there exists an exact sequence of $A$-modules
$0\to {_{A}A}\to T_{0}\to\cdots\to T_{\ell}\to 0$
with $T_{i}\in{\textrm{add}(_{A}T)}$ for all $0\leqslant i\leqslant \ell$. Since $M\in{^{\perp}(_{A}T)}$, there are isomorphisms $\Ext_{A}^{i}(\Omega^{\ell}_{A}(M),A)\cong\Ext_{A}^{\ell+i}(M,A)\cong\Ext_{A}^{i}(M,T_{\ell})=0$
for all $i\geqslant1$. This implies $M\in{\mathcal{GP}^{\leqslant \ell}(A)}$.
\end{proof}

The following corollary is an immediate consequence of Lemma \ref{prop:3.3}.

\begin{cor}\label{cor:3.5} Let $_{A}T$ be a Wakamatsu tilting module. Then:

(1) $^{\perp}(_{A}T)\cap{\mathcal{P}^{<\infty}(A)}=
\mathcal{W}{(_{A}T)}\cap{\mathcal{P}^{<\infty}(A)}$.

(2) If $_{A}T\in{\mathcal{P}^{<\infty}(A)}$, then
 $\mathcal{W}{(_{A}T)}\cap{\mathcal{GP}^{<\infty }(A)}={^{\perp}(_{A}T)}\cap{\mathcal{GP}^{<\infty}(A)}$.

(3) If $_{A}T$ is a tilting module of projective dimension $\ell$, then
$$\mathcal{W}{(_{A}T)}={^{\perp}(_{A}T)}\cap{\mathcal{GP}^{\leqslant \ell }(A)}={^{\perp}(_{A}T)}\cap{\mathcal{GP}^{<\infty}(A)},$$
$$\mathcal{W}{(T_B)}={^{\perp}(T_{B})}\cap{\mathcal{GP}^{\leqslant \ell}(B^{op})}={^{\perp}(T_{B})}\cap{\mathcal{GP}^{<\infty}(B^{op})}.$$
\end{cor}

\subsection{Proofs of Theorem \ref{THM} and Corollary \ref{GV}}

We begin this subsection with a proof of Theorem \ref{THM}(1).

{\bf Proof of Theorem \ref{THM}(1)}.  Recall from  \cite[Theorem 23.5]{Anderson} that, given inverse (not necessarily resolving) dualities $F:\mathcal{C}\to\mathcal{D}$ and $G:\mathcal{D}\to\mathcal{C}$ with ${_A}A\in\mathcal{C}$ and $B_B\in\mathcal{D}$, there exists a faithfully balanced bimodule $_{A}T_{B}$ with $T_{B}\cong{F(_{A}A)}\in{\mathcal{D}}$ and $_{A}T\cong{G(B_{B})}\in{\mathcal{C}}$ such that $F\cong\Hom_{A}(-,T)|_{\mathcal{C}}$ and  $G\cong\Hom_{B^{op}}(-,T)|_{\mathcal{D}}$. Now, assume that $F$ and $G$ are resolving dualities. They are exact functors between exact categories $\mathcal{C}$ and $\mathcal{D}$. Let $\delta: 0\to T\to X\to M\to0$ be an exact sequence in $A$-mod with $M\in{\mathcal{C}}$. Since $\mathcal{C}$ is closed under extensions in $A$-mod,  $X$ belongs to $\mathcal{C}$. As $F$ is exact, the sequence $F(\delta): 0\to F(M)\to F(X)\to F(T)\to 0$ is exact. It follows from $F(T)\cong {FG(B_{B})}\cong{B_{B}}$ that $F(\delta)$ splits in $B^{op}\modcat$. Note that $GF$ is naturally isomorphic to the identity functor of $\mathcal{C}$. Thus $\delta$ splits in $A$-mod. This implies that $\Ext_{A}^{1}(M,T)=0$ for all $M\in{\mathcal{C}}$. Since $\mathcal{C}\subseteq A\modcat$ is a resolving subcategory, $\Omega^{i}_{A}(M)\in{\mathcal{C}}$ for all $M\in{\mathcal{C}}$ and $i\geqslant1$. Consequently, $\Ext_{A}^{i}(M,T)\cong{\Ext_{A}^{1}(\Omega^{i-1}_{A}(M),T)=0}$. This yields $\mathcal{C}\subseteq{^{\perp}(_{A}T)}$. In particular, $_AT\in{^{\perp}(_{A}T)}$. Similarly, one can prove $\mathcal{D}\subseteq{^{\perp}(T_{B})}$, which gives rise to $T_B\in {^{\perp}(T_{B})}$. Since ${_A}T{_B}$ is faithfully balanced, it is Wakamatsu tilting.

To prove (c), it suffices to show  the inclusion $\mathcal{C}\subseteq{\mathcal{W}{(_{A}T)}}$, while the inclusion $\mathcal{D}\subseteq{\mathcal{W}{(T_{B})}}$ can be shown dually.

In fact, by $\mathcal{C}\subseteq{^{\perp}(_{A}T)}$,  we only need to show $\mathcal{C}\subseteq$ cogen$^{*}(_{A}T)$.
For each $X\in \mathcal{C}$, let
$$\cdots\to Q_{1}\s{g_{1}}\to Q_{0}\s{g_{0}}\to F(X)\to 0$$
be a projective resolution of $F(X)$ in $B^{op}\modcat$. Clearly, $F(X)\in{\mathcal{D}}$ and $Q_{i}\in{\mathcal{D}}$ for all $i\geqslant0$. Since $\mathcal{D}$ is a resolving subcategory of $B^{op}$-mod, $\textrm{Ker}(g_{i})\in{\mathcal{D}}$ for all $i\geqslant0$. By the exactness of $G$, we obtain an exact sequence of $A$-modules $0\to X\to G(Q_{0})\to G(Q_{1})\to \cdots$
in which $G(Q_{i})\in{\textrm{add}(G(B_{B}))=\textrm{add}(_{A}T)}$ for all $i\geqslant0$.
Thus the exactness of $F$ implies $X\in$ cogen$^{*}(_{A}T)$. This finishes the proof of Theorem \ref{THM}(1). \hfill$\Box$

\medskip
To show Theorem \ref{THM}(2), we first show the following result.

\begin{lem}\label{lem:3.6}  Let $_{A}T_{B}$ be the Wakamatsu tilting bimodule associated with inverse resolving dualities $F:\mathcal{C}\to \mathcal{D}$ and $G:\mathcal{D}\to \mathcal{C}$ in Theorem \ref{THM}(1).
Suppose that $\mathcal{U}\subseteq A\modcat$ and $\mathcal{V}\subseteq B^{op}\modcat$ are resolving subcategories.
Let $m$ and $n$ be natural numbers. Then:

$(1)$  $F$ and $G$ restrict to inverse dualities
    $\mathcal{C}\cap\mathcal{U}^{\leqslant n}(A)\simeq\mathcal{D}\cap\mathcal{V}^{\leqslant m}(B^{op})$
    if and only if the following condition holds:
$$(\ast)\quad  F(\mathcal{C}\cap\mathcal{U})\subseteq{\mathcal{D}\cap\mathcal{V}^{\leqslant m}(B^{op})}\quad\mbox{and}\quad G(\mathcal{D}\cap\mathcal{V})\subseteq{\mathcal{C}\cap\mathcal{U}^{\leqslant n}(A)}.$$

$(2)$ Assume that the condition $(\ast)$ holds. The following statements are true.

\quad\; $(a)$ $\mathcal{C}\cap{\mathcal{U}^{<\infty}(A)}=\mathcal{C}\cap{\mathcal{U}^{\leqslant n}(A)}$. If $\mathcal{U}\subseteq{\mathcal{C}}$, then $\mathcal{C}\cap{\mathcal{U}^{<\infty}(A)}={^{\perp}(_{A}T)\cap{\mathcal{U}^{<\infty}(A)}}.$

\quad\; $(b)$ $\mathcal{D}\cap{\mathcal{V}^{<\infty}(B^{op})}=\mathcal{D}\cap{\mathcal{V}^{\leqslant m}(B^{op})}$. If $\mathcal{V}\subseteq{\mathcal{D}}$, then $\mathcal{D}\cap{\mathcal{V}^{<\infty}(B^{op})}={^{\perp}(T_{B})\cap{\mathcal{V}^{<\infty}(B^{op})}}.$

\quad\; $(c)$ If $\mathcal{U}\subseteq{\mathcal{SGP}(A)}$ and $\mathcal{V}\subseteq{\mathcal{SGP}(B^{op})}$, then $_{A}T_{B}$ is a tilting bimodule.
\end{lem}

\begin{proof} (1) The necessity of $(1)$ is clear since $\mathcal{U}\subseteq \mathcal{U}^{\leqslant n}(A)$ and  $\mathcal{V}\subseteq \mathcal{V}^{\leqslant m}(B^{op})$. It suffices to show the sufficiency of $(1)$. Assume that the condition $(\ast)$ holds.
Let $M\in\mathcal{C}\cap\, \mathcal{U}^{\leqslant n}(A)$.
We need to show $F(M)\in{\mathcal{D}\cap\mathcal{V}^{\leqslant m}(B^{op})}$.

Let $\cdots\to P_1\to P_0\to M\to 0$ be a projective resolution of $_{A}M$ and let $s:=\textrm{max}\{m,n\}$. Since $\mathcal{C}$ and $\mathcal{U}$ are resolving subcategories of $A\modcat$, $\Omega^{s}_{A}(M)\in{\mathcal{C}\cap\mathcal{U}}$ by Lemma \ref{lemma:2.1}. It follows from $F(\mathcal{C}\cap\mathcal{U})\subseteq{\mathcal{D}\cap\mathcal{V}^{\leqslant m}(B^{op})}$ that $F(\Omega^{s}_{A}(M))\in{\mathcal{D}\cap\mathcal{V}^{\leqslant m}(B^{op})}$.
Since $M\in\mathcal{C}\subseteq{^{\perp}(_{A}T)}$ by Theorem \ref{THM}(1), there is an exact sequence in $B^{\rm op}\modcat$:
$$0\to F(M)\to F(P_0)\to F(P_1)\to\cdots \to F(P_{s-1})\to F(\Omega^{s}_{A}(M))\to 0.$$
Observe that both $\mathcal{D}$ and $\mathcal{V}^{\leqslant m}(B^{op})$ are resolving subcategories of $B^{op}\modcat$, and so is
their intersection ${\mathcal{D}\cap\mathcal{V}^{\leqslant m}(B^{op})}$. Since $F(P_i)\in F(\add(_{A}A))\subseteq F(\mathcal{C}\cap\mathcal{U})$ for all $0\leqslant i\leqslant s-1$, we have $F(M)\in{\mathcal{D}\cap\mathcal{V}^{\leqslant m}(B^{op})}$.

(2) It is enough to prove (a) and (c) since (b) can be shown dually.

(a) To prove $\mathcal{C}\cap{\mathcal{U}^{<\infty}(A)}=\mathcal{C}\cap{\mathcal{U}^{\leqslant n}(A)}$, it suffices to prove  $\mathcal{C}\cap{\mathcal{U}^{<\infty}(A)}\subseteq\mathcal{C}\,\cap{\mathcal{U}^{\leqslant n}(A)}$. Let $X\in{\mathcal{C}\cap{\mathcal{U}^{<\infty}(A)}}$. Then there is a natural number $t$ such that $\Omega^{t}_{A}(X)\in{\mathcal{C}\cap\mathcal{U}}$
by Lemma \ref{lemma:2.1}. It follows from the proof of the sufficiency of $(1)$ that  $F(X)\in{{\mathcal{D}\cap\mathcal{V}^{\leqslant m}(B^{op})}}$. Since $X\cong{GF(X)}$,  we have  $X\in{\mathcal{C}\cap{\mathcal{U}^{\leqslant n}(A)}}$.

Suppose $\mathcal{U}\subseteq{\mathcal{C}}$. By Theorem \ref{THM}(1), we only need to check ${^{\perp}(_{A}T)\cap{\mathcal{U}^{<\infty}(A)}}\subseteq{\mathcal{C}\cap{\mathcal{U}^{<\infty}(A)}}$. Let $Y\in{^{\perp}(_{A}T)\cap{\mathcal{U}^{<\infty}(A)}}$. By Lemma \ref{lemma:2.1}, we have an exact sequence in $A$-mod:
$$0\to U\to Q_{t-1}\to \cdots\to Q_1\to Q_0\to Y\to 0,$$
where $U\in{\mathcal{U}}$ and $Q_i\in\add(_{A}A)$ for all $0\leqslant i\leqslant{t-1}\in\mathbb{N}$. By $Y\in{^{\perp}(_{A}T)}$,  there is an exact sequence in
$B^{op}\modcat$:
$$0\to F(Y)\to F(Q_0)\to F(Q_1)\to \cdots\to  F(Q_{t-1})\to F(U)\to 0.$$
Clearly, $F(Q_{i})\cong\Hom_{A}(Q_{i},T)\in{\textrm{add}(T_{B})}$ for all $0\leqslant i\leqslant{t-1}$. Since  $\mathcal{U}\subseteq{\mathcal{C}}$, the condition $(\ast)$ implies $F(U)\in F(\mathcal{U})\subseteq{{\mathcal{D}\cap\mathcal{V}^{\leqslant m}(B^{op})}}$. Similarly, we can show $F(Y)\in{{\mathcal{D}\cap\mathcal{V}^{\leqslant m}(B^{op})}}$.
By the exactness of $G$, the following sequence
$$0\to GF(U)\to GF(Q_{t-1})\to \cdots\to GF(Q_1)\to GF(Q_0)\to GF(Y)\to 0$$
is exact in $A\modcat$. Recall that $GF$ is naturally isomorphic to the identity functor of $\mathcal{C}$.
Since $U\in\mathcal{U}\subseteq\mathcal{C}$ and $P_i\in\mathcal{C}$ for all $i$,  we have $Y\cong GF(Y)\in G(\mathcal{D})\subseteq{{\mathcal{C}}}$.

(c) Since $_{A}T\cong G(B_{B})$ by Theorem \ref{THM}(1), we obtain $_{A}T\in{\mathcal{C}\cap\mathcal{U}^{\leqslant{s}}(A)}$.  As $\mathcal{C}$ and $\mathcal{U}$ are resolving subcategories of $A\modcat$, it follows from Lemma \ref{lemma:2.1} that $\Omega^{i}_{A}(T)\in{\mathcal{C}\cap\mathcal{U}}$ for all $i\geqslant s$. In particular, $\Omega^{s+1}_{A}(T)\in {\mathcal{C}\cap\mathcal{U}}$. By  $(\ast)$, $F(\Omega^{s+1}_{A}(T))\in{\mathcal{D}\cap\mathcal{V}^{\leqslant s}(B^{op})}$.
Moreover, from the minimal projective resolution $\cdots\to P_2\to P_1\to P_0\to {_A}T\to 0$ of the module ${_A}T$, we obtain an exact sequence
$0\to B\to F(P_{0})\to  F(P_{1})\to F(P_2)\to \cdots$ in $B^{op}\modcat$, where $F(P_{i})\cong\Hom_{A}(P_{i},T)\in{\textrm{add}(T_{B})}$ for all $i\geqslant0$. Since $\mathcal{D}\subseteq {^{\perp}(T_{B})}$ by Theorem \ref{THM}(1)(c) and $F(\Omega^{s+1}_{A}(T))\in\mathcal{D}$, it is clear that
$\Ext_{B\opp}^{j}(F(\Omega^{s+1}_{A}(T)),F(P_{i}))=0$ for all $j\geqslant1$ and $i\geqslant0$.
Consequently, there are isomorphisms
$$\Ext_{B\opp}^{1}\big(F(\Omega^{s+1}_{A}(T)),F(\Omega^{s}_{A}(T))\big)\cong
\Ext_{B\opp}^{s+1}\big(F(\Omega^{s+1}_{A}(T)),B\big)\cong\Ext_{B\opp}^1\big(\Omega_{B^{op}}^s(F(\Omega^{s+1}_{A}(T))),B\big).$$
By Lemma \ref{lemma:2.1}, $\Omega_{B^{op}}^s(F(\Omega^{s+1}_{A}(T)))\in\mathcal{V}$. Since $\mathcal{V}\subseteq{\mathcal{SGP}(B^{op})}={^{\perp}(B_B)}$ by assumption,  $\Ext_{B\opp}^1\big(\Omega_{B^{op}}^s(F(\Omega^{s+1}_{A}(T))),B\big)=0$, and further $\Ext_{B\opp}^{1}\big(F(\Omega^{s+1}_{A}(T)),F(\Omega^{s}_{A}(T))\big)=0$
which leads to $F(\Omega^{s}_{A}(T))\in{\textrm{add}(T_{B})}$. Since $\Omega^{s}_{A}(T)\in \mathcal{C}$, we have $\Omega^{s}_{A}(T)\cong GF(\Omega^{s}_{A}(T))\in\add(_{A}A)$. This means that $\pd(_AT)\leqslant s$. Similarly, we can show $\pd(T_{B})\leqslant s$. Thus ${_A}T{_B}$ is tilting.
\end{proof}

{\bf Proof of Theorem \ref{THM}(2).}
The sufficiency of $(2)$ is a direct consequence of Lemma \ref{lem:3.6}.  Now, assume that $_{A}T_{B}$ is a tilting bimodule with
$\ell=\pd(_AT)$. By Theorem \ref{THM}(1) and Corollary \ref{cor:3.5}(3), $\mathcal{C}\subseteq\mathcal{W}({_A}T)\subseteq\mathcal{GP}^{\leqslant \ell}(A)$ and $\mathcal{D}\subseteq\mathcal{W}(T_B)\subseteq \mathcal{GP}^{\leqslant \ell}(B^{op})$.
It follows that $\mathcal{C}\cap\mathcal{GP}^{\leqslant \ell}(A)=\mathcal{C}$ and $\mathcal{D}\cap\mathcal{GP}^{\leqslant \ell}(B^{op}) =\mathcal{D}$.
Since $\mathcal{GP}^{\leqslant \ell}(A)\subseteq\mathcal{SGP}^{\leqslant \ell}(A)$, we also have $\mathcal{C}\cap\mathcal{SGP}^{\leqslant \ell}(A)=\mathcal{C}$ and $\mathcal{D}\cap\mathcal{SGP}^{\leqslant \ell}(B^{op})=\mathcal{D}$.
Thus the assertions on (semi-)Gorenstein-projective modules in the necessity of $(2)$ automatically hold.

Let $\mathcal{U}:=\add(_AA)$ and $\mathcal{V}:=\add(B_B)$.
Since $\mathcal{C}\subseteq A\modcat$ is a resolving subcategory, we see that
$F(M)\in{\textrm{add}(G(B_{B}))=\textrm{add}(_{A}T)}$ for any $M\in{\mathcal{U}\cap\mathcal{C}}$. This forces $F(M)\in{\mathcal{D}\cap\mathcal{V}^{\leqslant \ell}}(B^{op})$.  Dually, $G(N)\in{\mathcal{C}\cap\mathcal{U}^{\leqslant \ell}}(A)$ for any $N\in{\mathcal{V}\cap\mathcal{D}}$. By Lemma \ref{lem:3.6}(1),  $F$ and $G$ can be restricted to inverse resolving dualities
$\mathcal{C}\cap{\mathcal{P}^{\leqslant \ell}(A)}\simeq \mathcal{D}\cap{\mathcal{P}^{\leqslant \ell}(B^{op})}$.
Further, by Lemma \ref{lem:3.6}(2), $\mathcal{C}\cap{\mathcal{P}^{<\infty}(A)}=\mathcal{C}\cap{\mathcal{P}^{\leqslant \ell}(A)}$ and $\mathcal{D}\cap{\mathcal{P}^{\leqslant \ell}(B^{op})}= \mathcal{D}\cap{\mathcal{P}^{<\infty}(B^{op})}$. This finishes the proof of $(2)$.\hfill$\Box$

\smallskip
{\bf Proof of Corollary \ref{GV}.}
(1) follows from Corollary \ref{cor:3.5}(3) and the resolving dualities $\mathcal{W}{(_{A}T)}\simeq\mathcal{W}{(T_{B})}$ by \cite[Theorem 1.5(d)]{GRS}.

(2)  Suppose {$\mathcal{C}\subseteq {\mathcal{GP}^{\leqslant n }(A)}$ and $\mathcal{D}\subseteq {\mathcal{GP}^{\leqslant m}(B^{op})}$} for some $n, m\in\mathbb{N}$. Then $\mathcal{C}\cap\mathcal{GP}^{\leqslant n}(A)=\mathcal{C}$ and $\mathcal{D}\cap\mathcal{GP}^{\leqslant m}(B^{op}) =\mathcal{D}$. By Theorem \ref{THM}(2), ${_A}T_{B}$ is a tilting bimodule.
In Lemma \ref{lem:3.6}, we can take $\mathcal{U}:=\mathcal{GP}(A)$ and  $\mathcal{V}:=\mathcal{GP}(B^{op})$, and then
the condition $(\ast)$ is satisfied. Further, assume {$\mathcal{GP}(A)\subseteq\mathcal{C}$ and $\mathcal{GP}(B^{op})\subseteq\mathcal{D}$. By Lemma \ref{lem:3.6}(2), $\mathcal{C}={^{\perp}(_{A}T)\cap{\mathcal{GP}^{<\infty}(A)}}$ and $\mathcal{D}={^{\perp}(_{A}T)\cap{\mathcal{GP}^{<\infty}(B^{op})}}$.
Since $T$ is tilting, we see from Corollary \ref{cor:3.5}(3) that  $\mathcal{C}={^{\perp}(_{A}T)}\cap{\mathcal{GP}^{\leqslant \ell}(A)}$ and $\mathcal{D}={^{\perp}(T_{B})}\cap{\mathcal{GP}^{\leqslant \ell}(B^{op})}$. \hfill$\Box$

\medskip
Finally, we provide several equivalent characterizations for Wakamatsu tilting bimodules to be tilting. This may be helpful for understanding the Wakamatsu tilting conjecture from the viewpoint of resolving dualities.

\begin{cor}\label{cor:3.8} A Wakamatsu tilting bimodule $_{A}T_{B}$ is tilting if and only if
the functors $\Hom_{A}(-,T)$ and $\Hom_{B^{op}}(-,T)$ can be restricted to any one of inverse resolving dualities of the following types:

\begin{enumerate}
\item $\mathcal{W}{(_{A}T)}\cap{\mathcal{P}^{< \infty }(A)}\simeq \mathcal{W}{(T_{B})}\cap{\mathcal{P}^{<\infty}(B^{op})};$

\item  $^{\perp}(_{A}T)\cap{\mathcal{P}^{< \infty }(A)}\simeq {^{\perp}(T_{B})\cap{\mathcal{P}^{<\infty}(B^{op})}};$

\item $\mathcal{W}{(_{A}T)}\cap{\mathcal{GP}^{\leqslant n}(A)}\simeq \mathcal{W}{(T_{B})}\cap{\mathcal{GP}^{\leqslant m}(B^{op})}$
for some $n, m\in\mathbb{N}$;

\item $^{\perp}(_{A}T)\cap{\mathcal{GP}^{\leqslant n }(A)}\simeq {^{\perp}(T_{B})\cap{\mathcal{GP}^{\leqslant m}(B^{op})}}$
for some $n, m\in\mathbb{N}$;

\item $\mathcal{W}{(_{A}T)}\cap{\mathcal{SGP}^{\leqslant n}(A)}\simeq \mathcal{W}{(T_{B})}\cap{\mathcal{SGP}^{\leqslant m}(B^{op})}$
for some $n, m\in\mathbb{N}$.
\end{enumerate}
\end{cor}
\begin{proof}
Recall that the functors $F:=\Hom_{A}(-,T)$ and $G:=\Hom_{B^{op}}(-,T)$ between $A\modcat$ and $B^{op}\modcat$ can be restricted to inverse resolving dualities $\mathcal{W}{(_{A}T)}\simeq\mathcal{W}{(T_{B})}$ by \cite[Theorem 1.5(d)]{GRS}. So, the necessity of Corollary \ref{cor:3.8} follows from Theorem \ref{THM}(2) and Lemma \ref{prop:3.3}(1)(3). As to the sufficiency of Corollary \ref{cor:3.8}, all the cases except $(4)$ are direct consequences of Theorem \ref{THM}(2) and Lemma \ref{prop:3.3}(1). It remains to show the case $(4)$. Suppose that $F$ and $G$ are restricted to the inverse resolving dualities $^{\perp}(_{A}T)\cap{\mathcal{GP}^{\leqslant n }(A)}\simeq {^{\perp}(T_{B})\cap{\mathcal{GP}^{\leqslant m}(B^{op})}}$ for some $n, m\in\mathbb{N}$. By Theorem \ref{THM}(1), $^{\perp}(_{A}T)\cap{\mathcal{GP}^{\leqslant n }(A)}\subseteq{\mathcal{W}{(_{A}T)}}$ and
${^{\perp}(T_{B})\cap{\mathcal{GP}^{\leqslant m}(B^{op})}}\subseteq{\mathcal{W}{(T_{B})}}$.
This implies that $^{\perp}(_{A}T)\cap{\mathcal{GP}^{\leqslant n }(A)}=\mathcal{W}(_{A}T)\cap{\mathcal{GP}^{\leqslant n}(A)}$ and $^{\perp}(T_{B})\cap{\mathcal{GP}^{\leqslant n }(B^{op})}=\mathcal{W}(T_{B})\cap{\mathcal{GP}^{\leqslant n}(B^{op})}$, and then we return to the case $(3)$.\end{proof}

\begin{rem}\label{generalization}
In Theorem \ref{THM}, if $\mathcal{C}=A\modcat$ and $\mathcal{D}=B^{op}\modcat$, then ${_A}T$ and $T_B$ are injective cogenerators since  $A\modcat={\mathcal{W}{(_{A}T)}}$ and  $B^{op}\modcat={\mathcal{W}{(T_{B})}}$.
By Corollary \ref{cor:3.8}(2), we obtain Miyachita's duality (see Theorem \ref{thm:2.3}).
Further, if $F:\mathcal{C}\to \mathcal{D}$ is a resolving duality with $\mathcal{C} \subseteq \mathcal{P}^{<\infty}(A)$ and $\mathcal{D} \subseteq \mathcal{P}^{<\infty}(B^{op})$, then the conditions $(a)$ and $(b)$ in Huisgen-Zimmermann's correspondence (see Theorem \ref{thm:2.4}) also follow from Theorem \ref{THM} and Lemma \ref{lem:3.6}.
\end{rem}

\section{Applications to establish homological invariances under tilting}\label{K-theory}
In the section, we apply the resolving dualities established in the former section to homological invariants of algebras.
On the one hand, we employ Corollary \ref{GV} to construct triangle equivalences of derived categories of Gorenstein-projective modules and to show that higher algebraic $K$-groups of Gorenstein-projective modules are invariant under tilting. On the other hand, we show that semi-derived Ringel-Hall algebras of Gorenstein-projective modules over algebras are preserved under tilting (see Theorem \ref{thm:4.10}), which generalizes some results in \cite[A5]{Lu2022}.

\subsection{Derived equivalences and algebraic K-groups of
Gorenstein-projective modules}
Throughout the section, for a small exact category $\mathcal{E}$, we denote by
$K(\mathcal{E})$ the $K$-theory space of $\mathcal{E}$ in the sense of Quillen (see \cite{Quillen}), and by $K_{n}(\mathcal{E})$ the $n$-th homotopy group of $K(\mathcal{E})$ (called the $n$-th algebraic $K$-group of $\mathcal{E}$).

Let us recall two classical results on algebraic $K$-theory of exact categories. One is usually called the ``resolution theorem" (see, for example, \cite[Section 4]{Quillen}); the other conveys that algebraic $K$-groups of exact categories are invariant under dualities.

\begin{lem}\label{lemmma:5.1} Let $\mathcal{E}'$ be a full subcategory of a small exact category $\mathcal{E}$. Assume that the
following two conditions hold:

\begin{enumerate}
\item[(a)] If $X\rightarrowtail Y \twoheadrightarrow Z$ is a conflation in $\mathcal{E}$ with $Z\in{\mathcal{E}}$, then $Y\in{\mathcal{E}}$ if and only if $X\in{\mathcal{E}}$.

\item[(b)] For any object $M \in{\mathcal{E}}$, there is an exact sequence in $\mathcal{E}$:
$$0\to M_{n}\to M_{n-1}\to\cdots \to M_{1} \to M_{0}\to M\to 0$$
such that $M_{i}\in{\mathcal{E}}$  for all $0 \leqslant i \leqslant n$.
\end{enumerate}

\noindent Then the inclusion $\mathcal{E}'\subseteq \mathcal{E}$ of exact categories induces a homotopy equivalence of $K$-theory space
$K(\mathcal{E}')\s{\simeq}\longrightarrow K(\mathcal{E}).$ In particular, $K_n(\mathcal{E}')\simeq K_n(\mathcal{E})$ for all $n\in\mathbb{N}$.
\end{lem}

\begin{lem}\label{lemma:5.2} If $F:\mathcal{A}_{1}\to\mathcal{A}_{2}$ is a duality of small exact categories, then
$K_{n}(\mathcal{A}_{1})\simeq{K_{n}(\mathcal{A}_{2})}$ for all $ n\in{\mathbb{N}}$.
\end{lem}
\begin{proof} Since $F$ induces an equivalence $\mathcal{A}_{1}\s{\simeq}\longrightarrow {\mathcal{A}_{2}}^{op}$ of small exact categories, we obtain $K_{n}(\mathcal{A}_{1})\simeq{K_{n}({\mathcal{A}_{2}}\opp)}$ for each $n\in\mathbb{N}$. Now,
Lemma \ref{lemma:5.2} follows from $K_{n}(\mathcal{A}_{2})\simeq{K_{n}({\mathcal{A}_{2}}^{op})}$ (see \cite[Section 2]{Quillen}).
\end{proof}

{\bf Proof of Corollary \ref{Cor-1}(1)(2)}.
\begin{proof}
$(1)$ Let $\ell:=\pd(_AT)$. Set $\A_{1}:={^{\perp}(_{A}T)}\cap{\mathcal{GP}^{\leqslant\ell}(A)}$ and $\A_{2}:={^{\perp}(T_{B})}\cap{\mathcal{GP}^{\leqslant\ell}(B^{op})}$. By Corollary \ref{GV}(1), $F:=\Hom_{A}(-,T):\mathcal{A}_{1}\to\mathcal{A}_{2}$ is a duality of small exact categories. Let $G:=\Hom_{B\opp}(-,B):\mathcal{GP}(B^{op})\to \mathcal{GP}(B)$. Then $G$ is also a duality of small exact categories. Since $\mathcal{GP}(A)\subseteq \A_{1}$ and  $\mathcal{GP}(B^{op})\subseteq \A_{2}$, there is a diagram of exact categories:
\begin{equation}\label{diagram:4.0}
\begin{aligned}
\xymatrix@C=3em@R=3em{
 &\mathcal{GP}(A) \ar[r]^{\subseteq}& \A_{1}\ar[d]^{F}\\
\mathcal{GP}(B)& \mathcal{GP}(B^{op})\ar[l]_-{G}^-{\simeq}\ar[r]^{\subseteq}&\A_{2}.}
\end{aligned}
\end{equation}
Clearly, $\A_{1}\subseteq A\modcat$ and $\B_{1}\subseteq B^{op}\modcat$ are resolving subcategories.
Now, we apply Lemma \ref{lemma:2.3'} to the inclusions of exact categories
in \eqref{diagram:4.0}, and then obtain the following diagram of triangle equivalences of
unbounded derived categories of exact categories:
$$
\begin{aligned}
\xymatrix@C=3em@R=3em{
 &\mathscr{D}(\mathcal{GP}(A) )\ar[r]^{\;\simeq}&\mathscr{D}(\A_{1})\ar[d]^{\simeq}\\
\mathscr{D}(\mathcal{GP}(B))&\mathscr{D}(\mathcal{GP}(B^{op}))^{op}\ar[l]_{\simeq}\ar[r]^{\quad\simeq}&\mathscr{D}(\A_{2})^{op}.}
\end{aligned}
$$
This implies a triangle equivalence $\mathscr{D}(\mathcal{GP}(A))\simeq{\mathscr{D}(\mathcal{GP}(B))}$, which can be restricted to an equivalence $\mathscr{D}^{*}(\mathcal{GP}(A))\simeq\mathscr{D}^{*}(\mathcal{GP}(B))$ for any $*\in\{+,-,b\}$
by Lemma \ref{lemma:2.3'}.

$(2)$ By Lemma \ref{lemmma:5.1}, $K_{n}(\mathcal{GP}(A))\simeq{K_{n}(\A_{1})}$ and $K_{n}(\mathcal{GP}(B^{op}))\simeq{K_{n}(\A_{2})}$. By Lemma \ref{lemma:5.2}, $K_{n}(\A_{1})\simeq{K_{n}(\A_{2})}$ and $K_{n}(\mathcal{GP}(B))\simeq{K_{n}(\mathcal{GP}(B^{op}))}$. Thus $K_{n}(\mathcal{GP}(A))\simeq{K_{n}(\mathcal{GP}(B))}$.
\end{proof}

By use of resolving dualities, we can show the following result, of which all the assertions except the equivalence $\mathscr{K}(\mathcal{P}(A))\simeq{\mathscr{K}(\mathcal{P}(B))}$ are known (for example, see \cite{Happel1988,Rickard,Daniel}).

\begin{cor}\label{Projective} Let $A$ and $B$ be algebras and $_{A}T_{B}$ a tilting bimodule. Then:

\begin{enumerate}
\item There is a triangle equivalence $\mathscr{K}(\mathcal{P}(A))\simeq{\mathscr{K}(\mathcal{P}(B))}$ which
   can be restricted to an equivalence $\mathscr{K}^{*}(\mathcal{P}(A))\simeq \mathscr{K}^{*}(\mathcal{P}(B))$ for any
$*\in\{+,-,b\}$.

\item $K_{n}(\mathcal{P}(A))\simeq{K_{n}(\mathcal{P}(B))}$ for any $n\in{\mathbb{N}}.$

\end{enumerate}
\end{cor}
\begin{proof}
Note that $\mathscr{K}^{*}(\mathcal{P}(A))= {\mathscr{D}^{*}(\mathcal{P}(A))}$ for any $*\in\{\emptyset,+,-,b\}$.
Set $\B_{1}:={^{\perp}(_{A}T)}\cap{\mathcal{P}^{\leqslant\ell}(A)}$ and $\B_{2}:={^{\perp}(T_{B})}\cap{\mathcal{P}^{\leqslant\ell}(B^{op})}$ with $\ell:=\pd(_AT)$. Then $\B_{1}={^{\perp}(_{A}T)}\cap{\mathcal{P}^{<\infty}(A)}$ and $\B_{2}={^{\perp}(T_{B})}\cap{\mathcal{P}^{<\infty}(B^{op})}$ by Theorem \ref{thm:2.4}(c), and $F:\mathcal{B}_{1}\to\mathcal{B}_{2}$ is a duality of small exact categories by Theorem \ref{thm:2.4} (see also Corollary \ref{cor:3.8}(2)). Now, in the proof of Corollary \ref{Cor-1}(1)(2), we replace $\A_{1}$, $\A_{2}$, $\mathcal{GP}(A)$ and $\mathcal{GP}(B^{op})$ with $\B_{1}$ and $\B_{2}$, $\mathcal{P}(A)$ and $\mathcal{P}(B^{op})$, respectively, and then show Corollary \ref{Projective} similarly.
\end{proof}

\subsection{Semi-derived Ringel-Hall algebras of weakly $1$-Gorenstein exact categories}\label{Ringel-Hall-algebras}
In the section, we first recall the definition of semi-derived Ringel-Hall algebras of weakly $1$-Gorenstein exact categories from \cite{Lu2022}, and then
introduce a new definition for these algebras (up to isomorphism) which behaves better under dualities (see Proposition \ref{prop:4.6}).

Let $k:=\mathbb{F}_{q}$ be a finite field and $\mathcal{A}$ a small exact category linear over $k$. For each $X\in{\mathcal{A}}$, we define
\begin{center}
{${\rm Ext}\textrm{-}{\rm proj.dim}X:=\textrm{min}\{i\in{\mathbb{N}} \ | \ \textrm{Hom}_{\mathscr{D}(\mathcal{A})}(X,Y[j])=0\ \textrm{for} \ \textrm{all} \ Y\in{\A} \ \textrm{and} \ \textrm{all} \ j>i\}$,

\vspace{2mm}
${\rm Ext}\textrm{-}{\rm inj.dim}X:=\textrm{min}\{i\in{\mathbb{N}} \ | \ \textrm{Hom}_{\mathscr{D}(\mathcal{A})}(Y,X[j])=0\ \textrm{for} \ \textrm{all} \ Y\in{\A}\ \textrm{and} \ \textrm{all} \ j>i\}$}.
\end{center}

Following the Appendix of \cite{Lu2022}, there are four subcategories of $\mathcal{A}$:

\vspace{2mm}
\begin{center}{$\mathcal{P}^{\leqslant{i}}(\mathcal{A})=\{X\in{\mathcal{A}\ |\ {\rm Ext}\textrm{-}{\rm proj.dim}X\leqslant i}\}$,

\vspace{2mm}
$\mathcal{I}^{\leqslant{i}}(\mathcal{A})=\{X\in{\mathcal{A}\ |\ {\rm Ext}\textrm{-}{\rm inj.dim}X\leqslant i}\}$,

\vspace{2mm}
$\mathcal{P}^{<\infty}(\mathcal{A})=\{X\in{\mathcal{A}\ |\ {\rm Ext}\textrm{-}{\rm proj.dim}X<\infty}\}$,

\vspace{2mm}
$\mathcal{I}^{<\infty}(\mathcal{A})=\{X\in{\mathcal{A}\ |\ {\rm Ext}\textrm{-}{\rm inj.dim}X<\infty}\}$.}
\end{center}
The category $\mathcal{A}$ is said to be \emph{weakly Gorenstein} if $\mathcal{P}^{<\infty}(\A)=\mathcal{I}^{<\infty}(\A)$; \emph{weakly $d$-Gorenstein} if it is weakly Gorenstein and $\mathcal{P}^{<\infty}(\A)=\mathcal{P}^{\leqslant d}(\A)=\mathcal{I}^{\leqslant d}(\A)$.

We consider an exact category $\A$ satisfying the following conditions:
\begin{enumerate}
\item[(E-a)] $\A$ is a small exact category with finite morphism spaces and finite extension spaces, i.e.,
    $$|\Hom_{\A}(M,N)|<\infty, \quad |\Ext_{\A}^{1}(M,N)|<\infty;$$
\item[(E-b)] $\A$ is linear over $k=\mathbb{F}_{q}$;

\item[(E-c)] $\A$ is weakly 1-Gorenstein;

\item[(E-d)] for any object $X\in{\A}$, there exists a deflation $P_{X}\to X$ with  $P_{X}\in{\mathcal{P}^{<\infty}(\A)}$.
\end{enumerate}
In this case, $\mathcal{P}^{<\infty}(\A)=\mathcal{P}^{\leqslant 1}(\A)=\mathcal{I}^{<\infty}(\A)=\mathcal{I}^{\leqslant 1}(\A)$.

Clearly, if $A$ is a finite-dimensional algebra over $k$, then the Frobenius category $\mathcal{GP}(A)$ satisfies {\rm(E-a)-(E-d)}.
If, in addition, $A$ is $1$-Gorenstein (that is,  both $_AA$ and $A_A$ have injective dimension at most $1$), then the abelian category $A\modcat$ also satisfies {\rm(E-a)-(E-d)}. The following result supplies a class of weakly 1-Gorenstein exact categories which may be neither Frobenius nor abelian categories in general.

\begin{lem}\label{1-weak-G}
Let $A$ be a finite-dimensional algebra over the field $k$ and $_{A}T_{B}$ a tilting bimodule with $\pd(_AT)\leqslant 1$. Define
 $\A:={^{\perp}(_{A}T)}\cap{\mathcal{GP}^{\leqslant 1}(A)}$. Then $\A$ is weakly 1-Gorenstein satisfying  {\rm(E-a)-(E-d)}
 and  $\mathcal{P}^{<\infty}(\A)={^{\perp}(_{A}T)}\cap{\mathcal{P}^{\leqslant 1}(A)}$.
 \end{lem}

 \begin{proof}
By Corollary \ref{cor:3.5}(3), $\A=\mathcal{W}(_{A}T)$. Since $\A$ is a resolving subcategory of $A\modcat$ by Lemma \ref{proposition:3.2}(1), it is a small exact category of which the projective objects are exactly projective $A$-modules. By Lemma \ref{proposition:3.2}(3),  {\rm add}$(_{A}T)$ equals the full subcategory of $\A$ consisting of injective objects. So, $\A$ as an exact category has enough projective objects and injective objects. Clearly,
$\mathcal{P}^{<\infty}(\A)=\A\cap\mathcal{P}^{<\infty}(A)={^{\bot}(_{A}T)}\cap\mathcal{P}^{\leqslant 1}(A)$. Further, since  $\pd(_AT)\leqslant 1$, it follows from Theorem \ref{thm:2.4}(c) that ${^{\bot}(_{A}T)}\cap\mathcal{P}^{\leqslant 1}(A)={^{\bot}(_{A}T)}\cap\mathcal{P}^{<\infty}(A)$ which consists of those $A$-modules having finite $\add({_A}T)$-coresolutions. Thus $\mathcal{P}^{<\infty}(\A)=\mathcal{I}^{<\infty}(\A)$, that is, $\A$ is weakly Gorenstein.
Since $\mathcal{P}^{<\infty}(\A)\subseteq \mathcal{P}^{\leqslant 1}(A)$, we have
$\mathcal{P}^{<\infty}(\A)=\mathcal{P}^{\leqslant 1}(\A)$. Moreover, by the resolving duality $\Hom_A(-,T): {^{\bot}(_{A}T)}\cap\mathcal{P}^{\leqslant 1}(A)\to {^{\bot}(T_B)}\cap\mathcal{P}^{\leqslant 1}(B^{op})$ in Theorem \ref{thm:2.3}, each object of ${^{\bot}(_{A}T)}\cap\mathcal{P}^{\leqslant 1}(A)$ has an $\add({_A}T)$-coresolution of length at most $1$. This implies $\mathcal{I}^{<\infty}(\A)=\mathcal{I}^{\leqslant 1}(\A)$, and therefore $\A$ is weakly 1-Gorenstein.
Since $\A\subseteq A\modcat$ is a resolving subcategory and $k$ is a finite field, it can be checked that $\A$ satisfies {\rm(E-a)-(E-d)} .
 \end{proof}

Now, let $\A$ be an exact category which satisfies {\rm(E-a)-(E-d)}. Denote by ${\rm Iso}(\A)$ the set of isomorphism classes of objects in $\A$ and by $K_{0}(\A)$  the Grothendieck group of $\A$. Let $\HA$ be the \emph{Ringel-Hall algebra} of $\A$, that is, $\HA=\bigoplus_{[M]\in{\rm Iso}(\A)}\mathbb{Q}[M]$ as $\mathbb{Q}$-vector spaces with the multiplication given by
$$[M]\diamond[N]:=\sum_{[L]\in{{\rm Iso}(\A)}}\displaystyle\frac{|\Ext_{\A}^{1}(M,N)_{L}|}{|\Hom_{\A}(M,N)|}[L]$$
where $\Ext_{\A}^{1}(M,N)_{L}$ stands for the subset of ${\Ext_{\A}^{1}(M,N)}$ parameterizing all extensions in which the middle term is isomorphic to $L$.
Then $\HA$ is a $K_{0}(\A)$-graded algebra.
For $M\in{\A}$ and $K\in{\mathcal{P}^{\leqslant 1}(\A)}$, define
$$\langle K,M\rangle=\displaystyle\textrm{dim}_{k}\Hom_{\A}(K,M)-\displaystyle\textrm{dim}_{k}\Ext_{\A}^1(K,M),$$
$$\langle M,K\rangle=\displaystyle\textrm{dim}_{k}\Hom_{\A}(M,K)-\displaystyle\textrm{dim}_{k}\Ext_{\A}^1(M,K).$$
These formulas descend bilinear forms (called \emph{Euler forms}), again denoted
by $\langle \cdot,\cdot\rangle$, on the Grothendieck groups $K_{0}({\mathcal{P}^{\leqslant 1}(\A)})$ and $K_{0}(\A)$.

To introduce semi-derived Ringel-Hall algebras of weakly 1-Gorenstein exact categories, we
first recall the definition of (left or right) denominator subsets of rings and their relations with Ore localizations.

Let $R$ be a ring with identity and $S$ a subset of $R$ closed under multiplications with $1\in{S}$. Following \cite[Chapter 4]{Lam}, $S$ is called a \emph{left denominator subset} of $R$ if the following conditions hold:
\begin{enumerate}
\item[(i)] For any $a\in{R}$ and $s\in{S}$, the intersection $Sa\cap{Rs}$ is not empty;

\item[(ii)] For any $r\in{R}$, if $rt=0$ for some $t\in{S}$, there exists some $t'\in{S}$ such that $t'r=0$.
\end{enumerate}
If $S$ satisfies only the condition (i), then $S$ is called a \emph{left Ore subset} of $R$. Similarly, we can define right denominator sets and right Ore sets. Now, Ore's localization theorem states that
\begin{enumerate}
\item the left Ore localization  $[S^{-1}]R$  exists if and only if $S$ is a left denominator subset of $R$;
\item the right Ore localization  $R[S^{-1}]$  exists if and only if $S$ is a right denominator subset of $R$.
\end{enumerate}
If $S$ is a left and right denominator subset of $R$,
then $[S^{-1}]R$ is called the \emph{Ore localization} of $R$ at $S$.
In this case, up to isomorphism of rings, $[S^{-1}]R$, $R[S^{-1}]$ and the universal localization $R_{S}$ of $R$ at $S$ are
the same (for example, see \cite[Section 2.2]{chen-xi}).

Let $I(\A)$ be the two-sided ideal of $\HA$ generated by
$$\{[L]-[K\oplus M] \ | \ \exists \ \textrm{an exact sequence}\ 0\to K\to L\to M\to 0 \ \textrm{with}\ K\in{{\mathcal{P}^{\leqslant 1}(\A)}}\}.$$
We consider the following multiplicatively closed subset of the quotient $\HA/I(\A)$ of $\HA$ by $I(\A)$:
$$\mathcal{S}_{\A}:=\{a[K]\in{\HA/I(\A)} \ |\ a\in{\mathbb{Q}^{\times}},\ K\in{{\mathcal{P}^{\leqslant 1}(\A)}}\}.$$

\begin{lem}\label{lem:4.2}{\rm(}\cite[Proposition A5]{Lu2022}{\rm)}
$\mathcal{S}_{\A}$ is a right denominator subset of $\HA/I(\A)$. Equivalently, the right Ore localization $(\HA/I(\A))[\mathcal{S}^{-1}_{\mathcal{A}}]$ of $\HA/I(\A)$ with respect to $\mathcal{S}_{\A}$ exists.
\end{lem}

Following \cite{Lu2022} (also cf. \cite{Bridgeland, Gorsky, LP21}), the algebra $(\HA/I(\mathcal{A}))[\mathcal{S}^{-1}_{\mathcal{A}}]$ is called the \emph{semi-derived Ringel-Hall algebra} of $\A$ and denoted by $\mathcal{SDH}(\A)$.
Since the opposite category  $\A^{op}$ of $\A$ is also a weakly $1$-Gorenstein exact category, the algebra
$\mathcal{SDH}(\A^{op})$ is well defined. However, at the present time, it is not clear whether $\mathcal{SDH}(\A)\cong (\mathcal{SDH}(\A^{op}))^{op}$ as algebras because the definition of $\mathcal{SDH}(\A)$ seems not to be left-right symmetric.
To solve this problem, we will introduce a new definition of $\mathcal{SDH}(\A)$ up to isomorphism of algebras.

Let  $J(\A)$ be the two-sided ideal of $\HA$ generated by
$$\{[L]-[K\oplus M] \ | \ \exists \ \textrm{exact} \ \textrm{sequence}\ 0\to M\to L\to K\to 0 \ \textrm{with}\ K\in{{\mathcal{I}^{\leqslant 1}(\A)}}\}.$$
Then $I({\A}^{op})\cong{J(\A)}$ and there is an isomorphism of algebras:
$\HA/J(\A)\cong{(\mathcal{H}(\mathcal{A}^{op})/I(\mathcal{A}^{op}))^{op}}.$
Similarly, we consider the  multiplicatively closed subset $\mathcal{R}_{\A}$ of $\HA/J(\A)$:
$$\mathcal{R}_{\A}:=\{a[K]\in{\HA/J(\A)} \ |\ a\in{\mathbb{Q}^{\times}},\ K\in{{\mathcal{I}^{\leqslant 1}(\A)}}\}.$$
The following result is the dual of  Lemma \ref{lem:4.2}.

\begin{lem}\label{lem:4.4}
$\mathcal{R}_{\A}$ is a left denominator subset of $\HA/J(\A)$. Equivalently, the left Ore localization $[\mathcal{R}^{-1}_{\mathcal{A}}](\HA/J(\A))$ of
$\HA/J(\A)$ with respect to $\mathcal{R}_{\A}$ exists.  Moreover, there is an isomorphism $[\mathcal{R}^{-1}_{\mathcal{A}}](\HA/J(\A))\cong{(\mathcal{SDH}(\mathcal{A}^{op}))^{op}}$ of algebras.
\end{lem}

Now, we consider the quotient $\HA/(I(\A)+J(\A))$ of $\HA$ by the ideal $I(\A)+J(\A)$ and
its multiplicatively closed subset
$$\Phi_{\A}:=\{a[K]\in{\HA/(I(\A)+J(\A))} \ |\ a\in{\mathbb{Q}^{\times}},\ K\in{{\mathcal{P}^{\leqslant 1}(\A)}}\}.$$

\begin{lem}\label{lem:4.5}
$(1)$ $\Phi_{\A}$ is a left and right denominator subset of the algebra $\HA/(I(\A)+J(\A))$.

$(2)$ There are isomorphisms of algebras:

\begin{equation*}
\begin{aligned}
\mathcal{SDH}(\A)&\cong{(\HA/(I(\A)+J(\A)))[\Phi^{-1}_{\mathcal{A}}]}\\
&\cong{[\Phi^{-1}_{\mathcal{A}}](\HA/(I(\A)+J(\A)))}\\
&\cong{[\mathcal{R}^{-1}_{\mathcal{A}}](\HA/J(\A))}\\
&\cong{(\mathcal{SDH}(\mathcal{A}^{op}))^{op}}.
\end{aligned}
\end{equation*}
\end{lem}
\begin{proof} (1) Let $H:=\mathcal{H}(A)$, $I:=I(\A)$ and $J:=J(\A)$.
For any $K\in{{\mathcal{P}^{\leqslant 1}(\A)}}$ and $M\in{\A}$, it follows from \cite[Lemma A4]{Lu2022} that $[M]\diamond[K]=q^{-\langle M,K\rangle}[M\oplus K]$ in $H/I$. Dually, $[K]\diamond[M]=q^{-\langle K,M\rangle}[M\oplus K]$ in $H/J$. Thus $q^{\langle M,K\rangle}[M]\diamond[K]=q^{\langle K,M\rangle}[K]\diamond[M]$
in $H/(I+J)$. This implies that $\Phi_{\A}$ is a left and right Ore subset of the algebra $H/(I+J)$. By a similar argument as in the proof of \cite[Proposition A5]{Lu2022}, one can further show that $\Phi_{\A}$ is a left and right denominator subset of $H/(I+J)$.

(2) By $(1)$, there is an isomorphism of algebras
${(H/(I+J))[\Phi^{-1}_{\mathcal{A}}]}\cong{[\Phi^{-1}_{\mathcal{A}}](H/(I+J))}.$
By Lemma \ref{lem:4.4}, it is enough to show the algebra isomorphism $$\mathcal{SDH}(\A)\cong{(H/(I+J))[\Phi^{-1}_{\mathcal{A}}]}.$$ The algebra isomorphism ${[\Phi^{-1}_{\mathcal{A}}](H/(I+J))}\cong{[\mathcal{R}^{-1}_{\mathcal{A}}](H/J)}$ can be proved dually.

Let  $\lambda_{1}:H/I\to\mathcal{SDH}(\A)$ and $\lambda:H/(I+J)\to(H/(I+J))[\Phi^{-1}_{\mathcal{A}}]$
be the localizations and let  $\pi_{1}:H/I\to H/(I+J)$ be the canonical surjection.
Since $\pi_{1}(\mathcal{S}_{\A})=\Phi_{\A}$, there is a unique homomorphism of algebras
$$\sigma:\mathcal{SDH}(\A) \to{(H/(I+J))[\Phi^{-1}_{\mathcal{A}}]}$$ such that $\sigma\lambda_{1}=\lambda\pi_{1}$. By the statements following
\cite[Lemma A8]{Lu2022}, we have $\lambda_{1}(J)=0$ in $\mathcal{SDH}(\A)$. Then there exists a unique homomorphism of algebras $$\widetilde{\lambda_{1}}:{(H/(I+J))[\Phi^{-1}_{\mathcal{A}}]}\to \mathcal{SDH}(\A)$$ such that $\lambda_{1}=\widetilde{\lambda_{1}}\pi_{1}$, and hence $\widetilde{\lambda_{1}}([K])$ is invertible in $\mathcal{SDH}(\A)$ for any $[K]\in{\Phi_{\A}}$. This implies that $\lambda$ induces a unique homomorphism of algebras $$\rho:{(H/(I+J))[\Phi^{-1}_{\mathcal{A}}]}\to \mathcal{SDH}(\A)$$ such that $\rho\lambda=\widetilde{\lambda_{1}}$. So we have following commutative diagram:
$$
\xymatrix@C=3em@R=3em{
  H/I \ar[d]_{\pi_{1}} \ar[r]^{\lambda_{1}}
                & \mathcal{SDH}(\A)\ar@<.5ex>@{.>}[d]^{\sigma}   \\
  H/(I+J)\ar@{.>}[ur]^{\widetilde{\lambda_{1}}}\ar[r]_{\lambda}
                & (H/(I+J))[\Phi^{-1}_{\mathcal{A}}]\ar@<.5ex>@{.>}[u]^{\rho}.}
$$
Since $\rho\sigma\lambda_{1}=\rho\lambda\pi_{1}=\widetilde{\lambda_{1}}\pi_{1}=\lambda_{1}$, it follows from the the universal property of $\lambda_{1}$ that $\rho\sigma=\textrm{Id}$. On the other hand,  $\sigma\rho\lambda\pi_{1}=\sigma\widetilde{\lambda_{1}}\pi_{1}=\sigma\lambda_{1}=\lambda\pi_{1}$. Since $\pi_{1}$ is surjective, $\sigma\rho\lambda=\lambda$.  By the universal property of $\lambda$, we have $\sigma\rho=\textrm{Id}$. Thus $\sigma$ and $\rho$ are isomorphisms of algebras. \end{proof}

Thanks to Lemma \ref{lem:4.5}, up to isomorphism of algebras, we can define the \emph{semi-derived Ringel-Hall algebra}
of $\A$ to be the algebra $(\HA/(I(\A)+J(\A)))[\Phi^{-1}_{\mathcal{A}}]$. This definition is left-right symmetric and applied to show the following result.

\begin{prop}\label{prop:4.6} Let $F:\mathcal{A}_{1}\to\mathcal{A}_{2}$ be a duality of weakly $1$-Gorenstein exact categories. Then there exists an isomorphism of algebras

\vspace{2mm}
\begin{center}
{$\Upsilon_{F}:\mathcal{SDH}(\mathcal{A}_{1})\s{\simeq}\longrightarrow (\mathcal{SDH}(\mathcal{A}_{2}))^{op}$ \\
\vspace{2mm}
$[M]\mapsto [F(M)]$.}
\end{center}
\end{prop}

\begin{proof}
Note that $F$ induces an equivalence $\mathcal{A}_{1}\s{\simeq}\longrightarrow {\mathcal{A}_{2}}^{op}$ of weakly $1$-Gorenstein exact categories. Then $F$ induces an isomorphism of algebras: $\mathcal{SDH}(\mathcal{A}_{1})\s{\simeq}\longrightarrow{\mathcal{SDH}({\mathcal{A}_{2}}^{op})}.$ By Lemma \ref{lem:4.5}(2), ${\mathcal{SDH}({\mathcal{A}_{2}}^{op})}\simeq (\mathcal{SDH}(\mathcal{A}_2))^{op}$. Thus Proposition \ref{prop:4.6} holds.
\end{proof}

\subsection{Invariance of semi-derived Ringel-Hall algebras under tilting} \label{Ringel-Hall-algebras-Gorenstein}
In the section, our main result is the following theorem which contains Corollary \ref{Cor-1}(3) in the introduction.

\begin{thm}\label{thm:4.10} Let $A$ be a finite-dimensional algebra over $k$ and $_{A}T_{B}$  a tilting bimodule with $\pd(_AT)\leqslant1$. Then there exists an isomorphism of algebras:
\vspace{2mm}
\begin{center}
{$\Xi:\;\mathcal{SDH}(\mathcal{GP}(A))\s{\simeq}\longrightarrow \mathcal{SDH}(\mathcal{GP}(B))$ \\
\vspace{2mm}
$ [G]\mapsto q^{-\langle L,G\rangle}[\Hom_{A}(T,L)]^{-1}\diamond[\Hom_{A}(T,Z)]$,}
\end{center}
where $f:G\to Z$ is a  minimal left $(_{A}T)^{\perp}$-approximation of $G$ and
$L={\rm Coker}(f)$.
\end{thm}

When both $A$ and $B$ are $1$-Gorenstein algebras, Theorem \ref{thm:4.10} is exactly \cite[Corollary A23]{Lu2022}. To show Theorem \ref{thm:4.10} for general algebras, we establish a crucial result as follows.

\begin{prop}\label{prop:4.7} Let $A$ be a finite-dimensional algebra over $k$ and $_{A}T_{B}$ a tilting bimodule with $\pd(_AT)\leqslant1$. Set $\A:={^{\perp}(_{A}T)}\cap{\mathcal{GP}^{\leqslant 1}(A)}$ and $\B:=\mathcal{GP}(A)$. Then the embedding $\phi:\HB\to \HA$ induces an algebra isomorphism
$\widetilde{\phi}:\mathcal{SDH}(\B)\to\mathcal{SDH}(\A).$
Furthermore, the inverse of $\widetilde{\phi}$ is given by $\widetilde{\psi}:[M]\mapsto{q^{-\langle M,H_{M}\rangle}[G_{M}]\diamond[H_{M}]^{-1}}$, where $M\in{\A}$, $G_{M}\in{\B}$ and $H_{M}\in{\add(_{A}A)}$ such that they fit into a short exact sequence  $0\to H_{M}\to G_{M}\to M \to 0$ of $A$-modules.
\end{prop}

\begin{proof} Clearly, $\B$ is weakly $1$-Gorenstein satisfying  {\rm(E-a)-(E-d)}. By Lemma  \ref{1-weak-G} and its proof, $\A$ is also weakly $1$-Gorenstein satisfying  {\rm(E-a)-(E-d)} and ${\mathcal{P}^{{\leqslant 1}}(\A)}={^{\bot}(_{A}T)}\cap\mathcal{P}^{\leqslant 1}(A)$.  This means that $\mathcal{SDH}(\B)$ and $\mathcal{SDH}(\A)$ are well defined.
Moreover, since each object $M\in\A$ has Gorenstein dimension at most $1$, the exact sequence in Proposition \ref{prop:4.7} always exists. By a similar proof of \cite[Theorem A15]{Lu2022}, one can check that $\widetilde{\phi}$ is a surjective homomorphism of algebras and  the map $$\psi:\HA\to\mathcal{SDH}(\B),\;\;[M]\mapsto{q^{-\langle M,H_{M}\rangle}[G_{M}]\diamond[H_{M}]^{-1}}$$ is well defined.
Next we claim that $\psi$ is a homomorphism of algebras. It suffices to show that
\begin{equation}\label{4.0}\psi([M]\diamond[N])=\psi([M])\diamond\psi([N])
\end{equation}
for all $M,N\in{\A}$.

For this aim, we  fix two exact sequences in $A$-mod:
\begin{equation}\label{4.1}0\to H_{1}\to G_{1}\to M\to0,
\end{equation}
\begin{equation}\label{4.2} 0\to H_{2}\to G_{2}\to N\to0,
\end{equation}
where $H_{1},H_{2}\in{\add(_{A}A)}$ and $G_{1},G_{2}\in{\B}$. Applying $\Hom_{A}(-,N)$ and $
\Hom_{A}(G_1,-)$ to the sequences \eqref{4.1} and \eqref{4.2}, respectively,  we obtain the following diagram
\begin{equation}\label{4.3}
\begin{aligned}
\xymatrix{
  \Ext_{A}^{1}(M,N) \ar[dr] \ar[r]^{\Theta}
                & \Ext_{A}^{1}(G_{1},G_{2}) \ar[d]^{\cong}\ar[r] &0 \\
                & \Ext_{A}^{1}(G_{1},N). }
\end{aligned}
\end{equation}
Now, let $0\to N\to L\to M\to 0$ be an exact sequence in $\A$. By \eqref{4.3},  there is  commutative diagram with exact rows and columns:
\begin{equation}\label{4.4}
\begin{aligned}
\xymatrix{&0\ar[d]&0\ar[d]&0\ar[d]\\
0\ar[r]&H_2\ar[r]\ar[d]&H_{L}\ar[d]\ar[r]&H_1\ar[d]\ar[r]&0\\
0\ar[r]&G_2\ar[r]\ar[d]&G_{L}\ar[d]\ar[r]&G_1\ar[d]\ar[r]&0\\
0\ar[r]&N\ar[r]\ar[d]&L\ar[d]\ar[r]&M\ar[d]\ar[r]&0\\
&0&0&\;0.}
\end{aligned}
\end{equation}
Note that $\langle L,H_{L}\rangle={\langle M+N,H_{1}+H_{2}\rangle}={\langle M,H_{1}\rangle}+{\langle M,H_{2}\rangle}+{\langle N,H_{1}\rangle}+{\langle N,H_{2}\rangle}$. Since $\A$ and $\B$ are fully exact subcategories of $A\modcat$, there are equalities
\begin{center}
{$|\Ext_{\A}^{1}(X_1,X_2)|=|\Ext_{A}^{1}(X_1,X_2)|$, \quad $|\Hom_{\A}(X_1,X_2)|=|\Hom_{A}(X_1,X_2)|$,

\vspace{2mm}
$|\Ext_{\B}^{1}(Y_1,Y_2)|=|\Ext_{A}^{1}(Y_1,Y_2)|$, \quad $|\Hom_{\B}(Y_1,Y_2)|=|\Hom_{A}(Y_1,Y_2)|$ }
\end{center}
for all $X_1,X_2\in{\A}$ and $Y_1,Y_2\in{\B}$.
Thus
\begin{equation*}
\begin{split}
\psi([M]\diamond[N])&=\psi(\sum_{[L]\in{{\rm Iso}(\A)}}\displaystyle\frac{|\Ext_{A}^{1}(M,N)_{L}|}{|\Hom_{A}(M,N)|}[L])\\
&=\sum_{[L]\in{{\rm Iso}(\A)}}\displaystyle\frac{|\Ext_{A}^{1}(M,N)_{L}|}{|\Hom_{A}(M,N)|}\displaystyle q^{-\langle L,H_{L}\rangle}[G_{L}]\diamond[H_{L}]^{-1}\\
&=(\sum_{[L]\in{{\rm Iso}(\A)}}\displaystyle\frac{|\Ext_{A}^{1}(M,N)_{L}|}{|\Hom_{A}(M,N)|}q^{-\langle L,H_{L}\rangle}[G_{L}])\diamond[H_{1}\oplus H_{2}]^{-1}\\
&=(\sum_{[L]\in{{\rm Iso}(\A)}}\displaystyle\frac{|\Ext_{A}^{1}(M,N)_{L}|}{|\Hom_{A}(M,N)|}q^{-{\langle M,H_{1}\rangle}-{\langle M,H_{2}\rangle}-{\langle N,H_{1}\rangle}-{\langle N,H_{2}\rangle}}[G_{L}])\diamond[H_{1}\oplus H_{2}]^{-1}.
\end{split}
\end{equation*}
Since  $q^{\langle H_1,G_{2}\rangle}[H_1]\diamond[G_2]=[H_1\oplus G_2]=q^{\langle G_2,H_{1}\rangle}[G_2]\diamond[H_1]$, it follows that
\begin{equation*}
\begin{split}
\psi([M]) \diamond \psi([N])&=q^{-{\langle M,H_{1}\rangle}-{\langle N,H_{2}\rangle}}[G_1]\diamond{[H_1]^{-1}}\diamond[G_2]\diamond{[H_2]^{-1}}\\
&=q^{-{\langle M,H_{1}\rangle}-{\langle N,H_{2}\rangle}+{\langle H_{1},G_{2}\rangle}-{\langle G_{2},H_{1}\rangle}}[G_1]\diamond[G_2]\diamond{[H_1]^{-1}}\diamond{[H_2]^{-1}}\\
&=q^{-{\langle M,H_{1}\rangle}-{\langle N,H_{2}\rangle}+{\langle H_{1},G_{2}\rangle}-{\langle G_{2},H_{1}\rangle}}[G_1]\diamond[G_2]\diamond([H_2]\diamond[H_1])^{-1}\\
&=q^{-{\langle M,H_{1}\rangle}-{\langle N,H_{2}\rangle}+{\langle H_{1},G_{2}\rangle}-{\langle G_{2},H_{1}\rangle}+{\langle H_{2},H_{1}\rangle}}[G_1]\diamond[G_2]\diamond[H_2\oplus H_1]^{-1}.
\end{split}
\end{equation*}
Consequently, to prove \eqref{4.0}, we only need to check
\begin{align}\label{4.6}
q^{-{\langle M,H_{2}\rangle}-{\langle N,H_{1}\rangle}}\sum_{[L]\in{{\rm Iso}(\A)}}\displaystyle\frac{|\Ext_{A}^{1}(M,N)_{L}|}{|\Hom_{A}(M,N)|}[G_{L}]&
=q^{{\langle H_{1},G_{2}\rangle}-{\langle G_{2},H_{1}\rangle}+{\langle H_{2},H_{1}\rangle}}[G_1]\diamond[G_2].
\end{align}
Now, we set $K:=\ker(\Theta)$ in the diagram \eqref{4.3}. Then there exists an exact sequence
$$0\to \Hom_{A}(M,N)\to \Hom_{A}(G_1,N) \to \Hom_{A}(H_1,N)\to K\to0.$$
Thus
\begin{equation}\label{4.7}
|K|=q^{{\rm dim}_{k}\scriptstyle{\rm Hom}_{A}(H_1,N)-{\rm dim}_{k}\scriptstyle{\rm Hom}_{A}(G_1,N)+{\rm dim}_{k}\scriptstyle{\rm Hom}_{A}(M,N)}.
\end{equation}
For any $\delta':0\to G_2\to W\to G_1\to 0$ in the set $\Ext_{A}^{1}(G_1,G_2)_{W}$, one can show
$$|\{\delta\in{\Ext_{A}^{1}(M,N)}\ | \ \Theta(\delta)=\delta'\}|=|K|.$$ This leads to
$$|\{\delta\in{\Ext_{A}^{1}(M,N)}\ | \ \Theta(\delta)\in{\Ext_{A}^{1}(G_1,G_2)_{W}}\}|=|K||\Ext_{A}^{1}(G_1,G_2)_{W}|$$
and therefore \begin{equation}\label{4.8}
\sum_{[L]\in{{\rm Iso}(\A)}}|\Ext_{A}^{1}(M,N)_{L}|[G_{L}]=
|K|\sum_{[W]\in{{\rm Iso}(\A)}}|{\Ext_{A}^{1}(G_1,G_2)_{W}}|[W].
\end{equation}
Since $G_1\in\mathcal{B}$ and $H_{2}\in{\add(_{A}A)}$, the equality $\Ext_A^1(G_1, H_2)=0$ holds.
This implies that the sequence $0\to\Hom_{A}(G_1,H_{2})\to\Hom_{A}(G_1,G_{2}) \to \Hom_{A}(G_1,N)\to0$ is exact.
Consequently,
\begin{equation}\label{4.9}
\displaystyle{\rm dim}_{k}\Hom_{A}(H_1,N)-\displaystyle{\rm dim}_{k}\Hom_{A}(G_1,N)=
{\langle H_{1},N\rangle}+\displaystyle{\rm dim}_{k}\Hom_{A}(G_1,H_{2})-\displaystyle{\rm dim}_{k}\Hom_{A}(G_1,G_2).
\end{equation}
So we have

\begin{equation*}
\begin{aligned}
&\hspace{4mm}q^{-{\langle M,H_{2}\rangle}-{\langle N,H_{1}\rangle}}\sum_{[L]\in{{\rm Iso}(\A)}}\displaystyle\frac{|\Ext_{A}^{1}(M,N)_{L}|}{|\Hom_{A}(M,N)|}[G_{L}]\\
&=q^{-{\langle M,H_{2}\rangle}-{\langle N,H_{1}\rangle}+{\rm dim}_{k}\scriptstyle{\rm Hom}_{A}(H_1,N)-{\rm dim}_{k}\scriptstyle{\rm Hom}_{A}(G_1,N)}\sum\limits_{[W]\in{{\rm Iso}(\B)}}|\Ext_{A}^{1}(G_1,G_2)_{W}|[W]\quad(\mbox {by}\ \eqref{4.7},\;\eqref{4.8})\\
&=q^{-{\langle M,H_{2}\rangle}-{\langle N,H_{1}\rangle}+{\langle H_{1},N\rangle}+{\rm dim}_{k}\scriptstyle{\rm Hom}(G_1,H_{2})}\sum\limits_{[W]\in{{\rm Iso}(\B)}}\displaystyle\frac{|\Ext_{A}^{1}(G_1,G_2)_{W}|}{\Hom_{A}(G_1,G_2)}[W] \quad(\mbox {by}\ \eqref{4.9})\\
&=q^{-{\langle M,H_{2}\rangle}-{\langle N,H_{1}\rangle}+{\langle H_{1},G_2\rangle}-{\langle H_{1},H_2\rangle}+{\rm dim}_{k}\scriptstyle{\rm Hom}_{A}(G_1,H_{2})}[G_1]\diamond[G_2] \quad(\mbox {since}\ {\langle H_{1},G_2\rangle}={\langle H_{1},H_{2}+N\rangle})\\
&=q^{{\langle H_{1},G_2\rangle}-{\langle N,H_{1}\rangle}}q^{-({\langle M,H_{2}\rangle}+{\langle H_{1},H_2\rangle}-{\rm dim}_{k}\scriptstyle{\rm Hom}_{A}(G_1,H_{2}))}[G_1]\diamond[G_2] \\
&=q^{{\langle H_{1},G_2\rangle}-{\langle N,H_{1}\rangle}}[G_1]\diamond[G_2] \quad(\mbox {since}\ {\langle G_{1},H_2\rangle}={\langle M+H_{1},H_{2}\rangle})\\
&=q^{{\langle H_{1},G_2\rangle}-{\langle G_{2},H_{1}\rangle}+{\langle H_{2},H_{1}\rangle}}[G_1]\diamond[G_2] \quad(\mbox {since}\ {\langle G_{2},H_1\rangle}={\langle H_{2}+N,H_{1}\rangle}).
\end{aligned}
\end{equation*}
This shows that \eqref{4.6} is true, and thus \eqref{4.0} is also true. So, $\psi$ is a homomorphism of algebras.

Finally, we show that $\psi$ factorizes through the canonical surjection $\HA\to \HA/I(\A)$.

Suppose $N\in{\mathcal{P}^{{\leqslant 1}}(\A)}$. Since $H_2\in\add(_{A}A)$, we have $G_{2}\in{\mathcal{P}^{{\leqslant 1}}(A)}$ (see the first column in the diagram \eqref{4.4}). As $G_{2}$ lies in $\mathcal{B}$,  it is projective. This implies
$G_{L}\cong{G_{1}\oplus G_{2}}$, and therefore
$$\psi([L])={q^{-\langle L,H_{L}\rangle}[G_{L}]\diamond[H_{L}]^{-1}}=
{q^{-\langle M\oplus N,H_{1}\oplus H_{2}\rangle}[G_{1}\oplus G_{2}]\diamond[H_{1}\oplus H_{2}]^{-1}}=\psi([N\oplus M]).$$
Consequently, $\psi$ induces a homomorphism of algebras
$\psi':\HA/I(\A)\to\mathcal{SDH}(\B).$
Since $\psi([K])$ is invertible in $\mathcal{SDH}(\B)$ for any $K\in{\mathcal{P}^{{\leqslant 1}}(\A)}$, $\psi'$ induces a unique homomorphism of algebras
$\widetilde{\psi}:\mathcal{SDH}(\A)\to\mathcal{SDH}(\B).$
Clearly, $\widetilde{\psi}\widetilde{\phi}={\rm Id}$, which means that $\widetilde{\phi}$ is injective. Since $\widetilde{\phi}$ is surjective, it is an isomorphism of algebras.
\end{proof}

A  consequence of Proposition \ref{prop:4.7} is the following.

\begin{cor}\label{corollary:4.9} Let $A$ be a finite-dimensional algebra over $k$ and $_{A}T$ a strong tilting module with $\pd(_AT)\leqslant 1$. Then there exists an isomorphism of algebras:
$\mathcal{SDH}(\mathcal{GP}(A))\cong\mathcal{SDH}(\mathcal{GP}^{\leqslant 1}(A)).$
\end{cor}
\begin{proof} By Proposition \ref{prop:4.7}, it suffices to show ${^{\perp}(_{A}T)}\cap{\mathcal{GP}^{\leqslant 1}(A)}={\mathcal{GP}^{\leqslant 1}(A)}$. Let $M\in{\mathcal{GP}^{\leqslant 1}(A)}$. By \cite[Lemma 2.17]{cfh}, there exists an exact sequence $0\to M\to H\to G\to 0$ in $A$-mod with {\rm proj.dim$(_{A}H)\leqslant 1$} and $G\in{\mathcal{GP}(A)}$. It follows from Lemma \ref{prop:3.3}(2) that $G\in{\mathcal{W}{(_{A}T)}}$. Since $_{A}T$ is a strong tiling module with $\pd(_AT)\leqslant 1$, we have $H\in{^{\perp}(_{A}T)}$. Hence $M\in{^{\perp}(_{A}T)}$. This implies ${\mathcal{GP}^{\leqslant 1}(A)}\subseteq{{^{\perp}(_{A}T)}}$.
\end{proof}

{{\bf Proof of Theorem \ref{thm:4.10}}.}
\begin{proof}
Let $\A_{1}:={^{\perp}(_{A}T)}\cap{\mathcal{GP}^{\leqslant 1}(A)}$ and $\A_{2}:={^{\perp}(T_{B})}\cap{\mathcal{GP}^{\leqslant 1}(B^{op})}$. By Corollary \ref{GV}(1), $F:=\Hom_{A}(-,T):\mathcal{A}_{1}\to\mathcal{A}_{2}$ is a duality of weakly $1$-Gorenstein exact categories. Moreover, $G:=\Hom_{B\opp}(-,B):\mathcal{GP}(B^{op})\to \mathcal{GP}(B)$ is also a duality of weakly $1$-Gorenstein exact categories. By Propositions \ref{prop:4.7} and \ref{prop:4.6}, we obtain the following diagram of isomorphisms of algebras:
$$
\xymatrix@C=3em@R=3em{
  \mathcal{SDH}(\mathcal{GP}(A)) \ar[r]^-{\widetilde{\phi}_{A}}
                & \mathcal{SDH}(\A_{1})\ar[d]^{\Upsilon_{F}}  \\
                & (\mathcal{SDH}(\A_{2}))^{op}\ar[d]^{(\widetilde{\psi}_{B^{op}})^{op}}  \\
                 \mathcal{SDH}(\mathcal{GP}(B))
                & (\mathcal{SDH}(\mathcal{GP}(B^{op}))^{op}\ar[l]_-{(\Upsilon_{G})^{op}}}
$$
Define $$\Xi:={(\Upsilon_{G})^{op}}{(\widetilde{\psi}_{B^{op}})^{op}}{\Upsilon_{F}}{\widetilde{\phi}_{A}}: \;\mathcal{SDH}(\mathcal{GP}(A)) \to \mathcal{SDH}(\mathcal{GP}(B)).$$
Then $\Xi$ is an isomorphism of algebras. It remains to show that, for any $G\in{\mathcal{GP}(A)}$,
$$\Xi(G)=q^{-\langle L,G\rangle}[F(L)]^{-1}\diamond[F(Z)],$$
where $L$ and $Z$ are given in Theorem \ref{thm:4.10}.

In fact, since $G\in\mathcal{GP}(A)$, there is an exact sequence $0\to K\to P\to G\to 0$ in $A$-mod such that $P\in\add(_AA)$ and $K\in\mathcal{GP}(A)$. As $_{A}T$ is $1$-tilting, there is an exact sequence $0\to P\to T_{0}\to T_{1}\to0$ in $A$-mod with $T_{0},T_{1}\in{\add(_{A}T)}$. By pushout, we construct the diagram $(\sharp)$:
$$\xymatrix{&&0\ar[d]&0\ar[d]&\\
0\ar[r]&K\ar[r]\ar@{=}[d]&P\ar[r]\ar[d]&G
\ar[r]\ar[d]& 0\\
0\ar[r]&K\ar[r]&T_{0}\ar[d]\ar[r]&X
\ar[r]\ar[d]& 0\\
&&T_{1}\ar@{=}[r]\ar[d]&T_{1}\ar[d]&\\
&&0&\ 0.&\\
}$$
Since $T$ is $1$-tilting, $(_{A}T)^{\perp}={\rm Gen}(_AT)$,  the smallest full subcategory of $A\modcat$ containing $_AT$ and being closed under direct sums and quotients. This forces $X\in{(_{A}T)^{\perp}}$. Since $T_1\in\add(_AT)\subseteq (_{A}T)^{\perp}$, the third column in the diagram $(\sharp)$ implies that the map
$G\to X$ is a left  $(_{A}T)^{\perp}$-approximation of $G$. Now, let $f:G\to Z$ be a minimal left $(_{A}T)^{\perp}$-approximation with $L:={\rm Coker}(f)$. Then $f$ is injective; $Z$ and $L$ are isomorphic to direct summands of $X$ and $T_1$, respectively. In particular, $Z\in{\rm Gen}(_AT)$ and $L\in\add(_AT)$.   Since $\mathcal{GP}(A)\cup \add(_AT)\subseteq \A_1$ and $\A_1$ is closed under extensions in $A\modcat$, we have $Z\in\A_1$.  By $Z\in{\rm Gen}(_AT)$, there exists an exact sequence $0\to N\to H\to Z\to 0$ of $A$-modules with $H\in{\add(_{A}T)}$. Since $\A_1\subseteq A\modcat$ is a resolving subcategory, $N\in\A_{1}$. So, the sequence is exact in $\A_1$. Applying $F$ to the sequence yields an exact sequence $0\to F(Z)\to F(H)\to F(N)\to 0$ in $\A_2$. From $F(H)\in\add(B_B)$ and $F(N)\in{\mathcal{GP}^{\leqslant 1}(B^{op})}$, it follows that $F(Z)\in{\mathcal{GP}(B^{op})}$. Hence we have an exact sequence in $\A_2$
\begin{equation}\label{4.8.1}
0\longrightarrow F(L)\longrightarrow F(Z)\longrightarrow F(G)\longrightarrow 0
\end{equation}
such that $F(Z)$ is Gorenstein projective and $F(L)$ is projective. Moreover, there are equalities
\begin{equation}
\begin{aligned}\label{4.8.2}
\langle F(G),F(L)\rangle&=\displaystyle\textrm{dim}_{k}\Hom_{B\opp}(F(G),F(L))-
\displaystyle\textrm{dim}_{k}\Ext^{1}_{B}(F(G),F(L))\\
&=\displaystyle\textrm{dim}_{k}\Hom_{A}(L,G)-
\displaystyle\textrm{dim}_{k}\Ext^{1}_{B}(L,G)\quad(\mbox{by}\ \mbox{Lemma} \ \ref{lemma:3.1})\\
&=\langle L,G\rangle.
\end{aligned}
\end{equation}
Thus
\begin{equation*}
\begin{aligned}
\Xi([G])&=((\Upsilon_{G})^{op}}{(\widetilde{\psi}_{B^{op}})^{op}}{\Upsilon_{F}}{\widetilde{\phi}_{A})([G])\\
&={(\Upsilon_{G})^{op}}{(\widetilde{\psi}_{B^{op}})^{op}}([F(G)])\\
&={(\Upsilon_{G})^{op}}(q^{-\langle F(G),F(L)\rangle}[F(L)]^{-1}\diamond[F(Z)])\quad(\mbox{by}\ \eqref{4.8.1})\\
&=q^{-\langle F(G),F(L)\rangle}[GF(L)]^{-1}\diamond[GF(Z)]\\
&=q^{-\langle F(G),F(L)\rangle}[\Hom_{A}(T,L)]^{-1}\diamond[\Hom_{A}(T,Z)]\quad(\mbox{by}\ \mbox{Lemma} \ \ref{lemma:3.1})\\
&=q^{-\langle L,G\rangle}[\Hom_{A}(T,L)]^{-1}\diamond[\Hom_{A}(T,Z)]\quad(\mbox{by}\ \eqref{4.8.2}).
\end{aligned}
\end{equation*}
This finishes the proof of Theorem \ref{thm:4.10}.
\end{proof}

\begin{cor}\label{corollary:4.9'} Let $_{A}T_{B}$ be a tilting bimodule with $\pd(_AT)\leqslant 1$.
\begin{enumerate}
\item {\rm(}\cite[Theorem A15]{Lu2022}{\rm)} If $A$ is a finite-dimension 1-Gorenstein $k$-algebra, then there exists an isomorphism of algebras:
$\mathcal{SDH}(A)\cong\mathcal{SDH}(\mathcal{GP}(A)).$
\item {\rm(}\cite[Theorem A22]{Lu2022}{\rm)} If $A$ and $B$ are finite-dimension 1-Gorenstein $k$-algebras, then there exists an isomorphism of algebras: $\mathcal{SDH}(A)\cong\mathcal{SDH}(B)$.
\end{enumerate}
\end{cor}
\begin{proof}
Let $A$ be a finite-dimension 1-Gorenstein $k$-algebra. Then $A$-mod is a weakly 1-Gorenstein satisfying {\rm(E-a)-(E-d)}. Recall that $\mathcal{SDH}(A):=\mathcal{SDH}(A\modcat)$ is the semi-derived Ringel-Hall algebra of $A$.
Set $_{A}T:=\Hom_{k}(A,k)$, the ordinary injective cogenerator for $A\modcat$.
Then $_{A}T$ is a  strong tilting module and $\pd(_AT)\leqslant 1$. Moreover, $A\modcat={\mathcal{GP}^{\leqslant 1}(A)}={^{\perp}(_{A}T)}\cap{\mathcal{GP}^{\leqslant 1}(A)}$. Now, $(1)$ is a direct consequence of Corollary \ref{corollary:4.9} and (2) follows from $(1)$ and Theorem \ref{thm:4.10}.
\end{proof}

\begin{remark}\label{Different}
We point out that our proof of Corollary \ref{corollary:4.9'} is different from the proof given in \cite{Lu2022}.

$(a)$ In the proof of \cite[Theorem A15]{Lu2022}, under the assumption that $A$ is $1$-Gorenstein, the map $\psi:\mathcal{H}(\A)\to\mathcal{SDH}(\mathcal{GP}(A))$ (see the proof of Proposition \ref{prop:4.7} for $T:=\Hom_{k}(A,k)$) was shown to induce a unique morphism of $\mathcal{T}(A)$-bimodules $\widetilde{\psi}:\mathcal{SDH}(\A)\to\mathcal{SDH}(\mathcal{GP}(A))$ by using an explicit description of $\mathcal{SDH}(\A)$ as a $\mathcal{T}(A)$-bimodule (see \cite[Proposition A13]{Lu2022}), where $\A=A\modcat$ and $\mathcal{T}(A):=\mathcal{SDH}(\mathcal{P}^{\leqslant 1}(A))$ is a subalgebra of $\mathcal{SDH}(A)$. In our proof, we have shown that $\psi:\mathcal{H}(\A)\to\mathcal{SDH}(\mathcal{GP}(A))$ for a general $1$-tilting module ${_A}T$ (see the proof of Proposition \ref{prop:4.7}) is an algebra homomorphism and automatically induces an algebra homomorphism $\widetilde{\psi}:\mathcal{SDH}(\A)\to\mathcal{SDH}(\mathcal{GP}(A))$
which is the inverse of $\widetilde{\phi}$. Moreover, the proof of Proposition \ref{prop:4.7} does not involve the structure of $\mathcal{SDH}(\A)$ as a $\mathcal{T}(A)$-bimodule.

$(b)$ In the proof of \cite[Theorem A22]{Lu2022}, when both $A$ and $B$ are $1$-Gorenstein, the additive equivalence between the coresolving subcategory $({_A}T)^{\perp}:=\{X\in A\modcat\mid \Ext_A^1(T,X)=0\}$ of $A\modcat$ and the resolving subcategory $\mathcal{Y}:=\{Y\in B\modcat\mid{\rm Tor}_1^B(T,Y)=0\}$ of $B\modcat$ (by the Brenner-Butler tilting theorem) was applied to establish the second isomorphism in the following algebra isomorphisms:
$$\mathcal{SDH}(A)\cong\mathcal{SDH}((_{A}T)^{\perp})\cong
\mathcal{SDH}(\mathcal{Y})\cong\mathcal{SDH}(B),$$
where $(_{A}T)^{\perp}$ and  $\mathcal{Y}$ are weakly $1$-Gorenstein exact categories, and the first isomorphism is induced from the inclusion $({_A}T)^{\perp}\subseteq A\modcat$ (see \cite[Proposition A21]{Lu2022}). However, for a general algebra $A$, the category $(_{A}T)^{\perp}$ may not be weakly $1$-Gorenstein, and therefore $\mathcal{SDH}((_{A}T)^{\perp})$ is not well defined in general.
In our proof of Corollary \ref{corollary:4.9'}(2), we use the resolving dualities in Corollary \ref{GV}(1) and establish a series of algebra isomorphisms:
$$\mathcal{SDH}(A)\cong\mathcal{SDH}(\mathcal{GP}(A))\cong\mathcal{SDH}(\mathcal{GP}(B))\cong\mathcal{SDH}(B).$$
\end{remark}

\smallskip
{\bf Acknowledgements.} The research was supported by the National Natural Science Foundation of China (Grant 12122112, 12031014 and 12171206). Also, the author J. S. Hu thanks the Natural Science Foundation of Jiangsu Province (Grant No. BK20211358) and Jiangsu 333 Project for partial support.

\vspace{4mm}
\noindent\textbf{Hongxing Chen}\\
School of Mathematical Sciences \&  Academy for Multidisciplinary Studies, \\Capital Normal University, Beijing 100048, P. R. China;\\
Email: \textsf{chenhx@cnu.edu.cn}\\[1mm]
\textbf{Jiangsheng Hu}\\
School of Mathematics and Physics, Jiangsu University of Technology, Changzhou 213001,\\ P. R. China.\\
Email: \textsf{jiangshenghu@hotmail.com}\\[1mm]
\end{document}